\documentclass[10pt,twoside]{amsart}
\usepackage{amsmath, amsthm, amscd, amsfonts, amssymb, graphicx}
\newtheorem{thm}{Theorem}[section]
\newtheorem{cor}[thm]{Corollary}
\newtheorem{lem}[thm]{Lemma}

\newtheorem{defn}[thm]{Definition}
\newtheorem{rem}[thm]{\bf{Remark}}

\numberwithin{equation}{section}

\begin{document}
\begin{center}
{\bf{Some  Fibonacci sequence spaces of non-absolute type derived from $\ell_{p} $  with $(1 \leq p \leq \infty)$\\ and Hausdorff measure of non-compactness of composition operators}}
\vspace{.5cm}

Anupam Das$^{1}$, Bipan Hazarika$^{1,\ast}$ and Feyzi Ba\c{s}ar$^{2}$

\vspace{.2cm}

$^{1}$Department of Mathematics, Rajiv Gandhi University, Rono Hills,\\ Doimukh-791 112, Arunachal Pradesh, India\\
$^{2}$ K\i s\i kl\i{} Mah. Alim Sok. No:7/6, \"Usk\"udar/\.Istanbul, Turkey\\
Email: anupam.das@rgu.ac.in; bh\_rgu@yahoo.co.in; feyzibasar@gmail.com
\end{center}

\title{}
\author{}
\thanks{{\today\\ $^{\ast}$The corresponding author}}

\begin{abstract} The aim of the paper is to introduce the spaces $\ell_{\infty}^{\lambda}(\widehat{F})$ and $\ell_{p}^{\lambda}(\widehat{F})$ derived
by the composition of the two infinite matrices $\Lambda=(\lambda_{nk})$ and $\widehat{F}=\left( f_{nk} \right),$ which are the $BK$-spaces of non-absolute type and also derive some inclusion relations. Further, we determine the $\alpha$-, $\beta$-, $\gamma$-duals of those spaces and also construct the basis for $\ell_{p}^{\lambda}(\widehat{F}).$ Additionally, we characterize some matrix classes on the spaces $\ell_{\infty}^{\lambda}(\widehat{F})$ and $\ell_{p}^{\lambda}(\widehat{F}).$ We also investigate some geometric properties concerning Banach-Saks type $p.$
Here we characterize the subclasses $\mathcal{K}(X:Y)$ of compact operators, where $X\in\{\ell_{\infty}^{\lambda}(\widehat{F}),\ell_{p}^{\lambda}(\widehat{F})\}$  and $Y\in\{c_{0},c, \ell_{\infty}, \ell_{1}, bv\}$ by applying the Hausdorff measure of non-compactness, and $1\leq p<\infty.$
\vskip 0.5cm

\textbf{Key words:} Fibonacci numbers; $\alpha$-,$\beta$-,$\gamma$-duals; Matrix Transformations; Measure of non-compactness;
Hausdorff measure of non-compactness; Compact operator; Fixed point property; Banach-Saks type $p.$
\vskip 0.2cm

\textbf{2010 Mathematics Subject Classification:} 11B39; 46A45; 46B20; 46B45.

\end{abstract}
\maketitle
\pagestyle{myheadings}
\markboth{\rightline {\scriptsize Das, Hazarika and Ba\c{s}ar}}{\leftline{\scriptsize  Fibonacci sequence spaces of non-absolute type ... }}

\section{ Introduction}

Define the sequence $(f_{n})_{n\in\mathbb{N}}$ of Fibonacci numbers given by the linear recurrence relations $ f_{0}=f_{1}=1$ and $ f_{n}=f_{n-1}+f_{n-2}, n\geq 2,$ where $\mathbb{N}_0=\{0,1,2,\ldots\}.$ Fibonacci numbers have many interesting properties and applications. For example, the ratio sequences of Fibonacci numbers converges to the golden ratio which is important in sciences and arts. Also, some basic properties of Fibonacci numbers are given as follows:
\begin{eqnarray*}
\begin{array}{ll}
\underset{n\rightarrow\infty}{\lim}\dfrac{f_{n+1}}{f_{n}}= \dfrac{1+\sqrt{5}}{2}=\alpha\quad
(\textrm{golden ratio}), &\stackrel{n}{\underset{k=0}{\sum}}f_{k}=f_{n+2}-1 \quad (n\in \mathbb{N}_0), \\
\underset{k}{\sum}\dfrac{1}{f_{k}} ~\mbox{\textrm{ converges}}, &f_{n-1}f_{n+1}-f_{n}^{2}=(-1)^{n+1} \quad (n \geq 1)  ~(\textrm{Cassini formula})
\end{array}
\end{eqnarray*}
Substituting for $ f_{n+1} $ in Cassini's formula yields $ f_{n-1}^{2}+f_{n}f_{n-1}-f_{n}^{2}=(-1)^{n+1}$, (see \cite{Koshy}).

Let $\omega$ be the space of all real-valued sequences. Any vector subspace of $\omega$ is called a $\mathit{sequence\ space}.$ By $\ell_{\infty},$ $c,$ $c_{0}$ and $\ell_{p},$ we denote the sets of all bounded, convergent, null and $p$-absolutely summable sequences, respectively. Here and after, we suppose unless stated otherwise that $1\leq p<\infty$ and $q=p/(p-1).$ Also, we use the conventions that $e=(1,1,\ldots)$ and $e^{(n)}$ is the sequence whose only non-zero term is 1 in the $nth$ place for each $n\in\mathbb{N}_0.$

Let $X$ and $Y$ be two sequence spaces, and $A=(a_{nk})$ be an infinite matrix of real numbers $a_{nk},$ where $n,k\in\mathbb{N}_0.$ Then we say that $A$ defines a matrix mapping from $X$ into $Y$ and we denote it by writing $A:X\rightarrow Y$ if for every sequence $x=(x_{k})_{k\in\mathbb{N}_0}\in X,$ the sequence $Ax=\left\lbrace A_{n}(x) \right\rbrace_{n\in\mathbb{N}_0},$ the $A$-transform of $x,$ is in $Y,$ where
\begin{eqnarray}
 A_{n}(x)=\sum_{k} a_{nk}x_{k}~\textrm{ for each }~n\in \mathbb{N}_0.
 \label{1}
\end{eqnarray}
For simplicity in notation, here and in what follows, the summation without limits runs from $0$ to $\infty.$

By $(X:Y),$ we denote the class of all matrices $A$ such that $A:X\rightarrow Y.$ Thus $A\in (X:Y)$ iff the series on the right-hand side of (\ref{1}) converges for each $n\in\mathbb{N}_0 $ and every $x\in X,$ and we have $Ax\in Y$ for all  $x\in X.$

A sequence space $X$ is called an $FK$-space if it is complete linear metric space with continuous coordinates $ p_{n}:X\rightarrow \mathbb{R},$ where $ \mathbb{R}$ denotes the real field and $p_{n}(x)=x_{n}$ for all $x=(x_{n})\in X $ and every $n\in\mathbb{N}_0.$ A $BK$-space is a normed $FK$-space, that is, a $BK$-space is a Banach space with continuous coordinates. $\ell_{p}$ is a $BK$-space with the norm
\[ \| x \|_{p} =\left(\sum_{k}| x_{k}| ^{p}\right)^{1/p} \]
and $ c_{0},c $ and $ \ell_{\infty} $ are $BK$-spaces with the norm $\| x \|_{\infty}=\sup\limits_{k\in\mathbb{N}_0}| x_{k}|.$

A sequence $(b_{n})$ in a normed space $X$ is called a \textit{Schauder basis} for $X$ if every $x\in  X,$ there is a unique sequence $(\alpha_{n})$
of scalars such that $x=\sum\limits_{n}\alpha_{n}b_{n},$ i.e., $\left\|x-\sum\limits_{n=0}^{m}\alpha_{n}b_{n}\right\|\to 0,$ as $m\to\infty.$

The matrix domain plays an important role to construct a new sequence space. In studies on the sequence spaces, generally there are some approaches. Most important of them are determination of topologies, matrix
mappings and inclusion relations.
The matrix domain $X_{A}$ of an infinite matrix $A$ in a sequence space $X$ is defined by
\begin{eqnarray}\label{matdom}
X_{A}=\left\lbrace x=(x_{k})\in \omega: Ax \in X\right\rbrace.
\end{eqnarray}
It is easy to see that $X_{A}$ is a sequence space whenever $X$ is a sequence space. In the past, several authors studied matrix transformations on sequence spaces that are the matrix domain of the difference operator, or of the matrices of some classical methods of summability in different sequence spaces, for instance we refer to \cite{fb,hcfb,dashazarika1,dashazarika2,Kara16, Karamursaleen, Kara15,Kiri, cp2,M.Mursaleen,tripathy,otfb,otfb2} and references therein. The Hausdorff measure of non-compactness of linear operators given by infinite matrices in some special classes of sequence spaces were studied in \cite{abdullah, fbem,cp4,mursallenpiloat,leen}.

The  $\alpha$-, $\beta$- and $\gamma$-duals $X^{\alpha},$ $X^{\beta}$ and $X^{\gamma}$ of a sequence space $X$ are respectively defined by
\begin{eqnarray*}
X^{\alpha}&:=&\left\lbrace a=(a_{k})\in \omega : ax=(a_{k}x_{k})\in \ell_{1}~\textrm{ for all }~x=(x_{k})\in X  \right\rbrace,\\
X^{\beta}&:=&\left\lbrace a=(a_{k})\in \omega : ax=(a_{k}x_{k})\in cs~\textrm{ for all }~x=(x_{k})\in X  \right\rbrace, \\
X^{\gamma}&:=&\left\lbrace a=(a_{k})\in \omega : ax=(a_{k}x_{k})\in bs~\textrm{ for all }~x=(x_{k})\in X  \right\rbrace,
\end{eqnarray*}
where $cs$ and $bs$ are the spaces of all convergent and bounded series, respectively (see \cite{A,PK,M}).

If $X \supset \phi$ is a $BK$-space and $a=(a_{k})\in\omega,$ then we write \[\|a\|_{X}^{*}=\sup\left\lbrace
\left| \sum_{k} a_{k} x_{k}\right|  : \| x \|=1\right\rbrace .\]

Let $X$ and $Y$ be Banach spaces. A linear operator $L : X \rightarrow Y$ is called compact if its domain is all of $X$ and for every bounded sequence $(x_{n})$ in $X,$ the sequence $(L(x_{n}))$ has a convergent subsequence in $Y.$ We denote the class of compact operators by $\mathcal{K}(X:Y).$

Let us recall some definitions and well-known results.
\begin{defn}
Let $(X,d)$ be a metric space, $Q$ be a bounded subset of $X$ and $B(x,r)=\left\lbrace y\in X : d(x,y)<r\right\rbrace.$ Then, the Hausdorff measure of non-compactness $\chi(Q)$ of $Q$ is defined by
\[\chi(Q):=\inf \left\lbrace \epsilon>0: Q \subset \bigcup_{i=1}^{n} B(x_{i},r_{i}),x_{i}\in X,r_{i}< \epsilon \quad (i=1,2,\ldots,n),n\in\mathbb{N}_0  \right\rbrace.\]
\end{defn} Then, the following results can be found in \cite{Bana, cp3}.
\par If $Q,Q_{1}$ and $Q_{2}$ are bounded subsets of the metric space $(X,d),$ then we have
\begin{eqnarray*}
&&\chi(Q)=0 \mbox{~if~ and~ only~ if~} Q \mbox{~is~a ~totally ~bounded~ set},\\
&&\chi(Q)=\chi(\bar{Q}), \\
&&Q_{1}\subset Q_{2} \mbox{~implies~} \chi(Q_{1})\leq \chi(Q_{2}), \\
&&\chi(Q_{1} \cup Q_{2})=\max\left\lbrace \chi(Q_{1}),\chi(Q_{2})\right\rbrace,\\
&&\chi(Q_{1} \cap Q_{2})\leq \min\left\lbrace \chi(Q_{1}),\chi(Q_{2})\right\rbrace.
\end{eqnarray*}
If $Q,Q_{1}$ and $Q_{2}$ are bounded subsets of the normed space $X,$ then we have
\begin{eqnarray*}
&&\chi(Q_{1}+Q_{2})\leq \chi(Q_{1})+\chi(Q_{2}), \\
&&\chi(Q+x)=\chi(Q) \mbox{~for~ all~} x \in X,  \\
&&\chi(\lambda Q)=\left|\lambda \right| \chi(Q)\mbox{~ for~ all~} \lambda \in \mathbb{C}.
\end{eqnarray*}
\begin{defn}
 Let $X$ and $Y$ be Banach spaces and $\chi_{1}$ and $\chi_{2}$ be Hausdorff measures on $X$ and $Y.$ Then, the operator $L:X\rightarrow Y$ is called $\left( \chi_{1}:\chi_{2}\right)$-bounded  if $L(Q)$ is bounded subset of $Y$ for every subset $Q$ of $X$ and there exists a positive constant $K$ such that $\chi_{2}(L(Q))\leq K \chi_{1}(Q)$ for every bounded subset $Q$ of $X.$ If an operator $L$ is $(\chi_{1}:\chi_{2})$-bounded, then the number
\begin{eqnarray*}
\|L\|_{(\chi_{1}:\chi_{2})}=\inf \left\lbrace K>0:\chi_{2}(L(Q))\leq K \chi_{1}(Q)~\textrm{ for all bounded }~Q \subset X\right\rbrace
\end{eqnarray*}
is called $(\chi_{1}:\chi_{2})$- measure of non-compactness of $L$. In particular, if $\chi_{1}=\chi_{2}=\chi,$ then we write $\| L \|_{\chi}$ instead of
 $\| L\|_{(\chi:\chi)}.$
 \end{defn}

The idea of compact operators between Banach spaces is closely related to the Hausdorff measure of non-compactness, and it can be given, as follows: Let $X$ and $Y$ be Banach spaces and $L \in B(X:Y).$ Then, the Hausdorff measure of non-compactness $\|L\|_{\chi}$ of $L$ can be given by $\|L\|_{\chi}=\chi(L(S_{X})),$ where $S_{X}=\left\lbrace x\in X : \| x\|=1\right\rbrace $
 and we have  $L$ is compact if and only if $\| L\|_{\chi}=0.$ We also have $\|L\|=\sup\limits_{x \in S_{X}}\|Lx\|_{Y}.$

\section{The sequence spaces $\ell_{p}^{\lambda}(\widehat{F}),$  $(1\leq p \leq \infty) $ of non-absolute type }
Das and Hazarika \cite{dashaza01} introduced the spaces $c_{0}^{\lambda}(\widehat{F})$ and  $c^{\lambda}(\widehat{F})$ derived
by the composition of the two infinite matrices $\Lambda$ and $\widehat{F},$ and obtain some interesting results in terms of the domain of the product of two infinite matrices.

In this section, we introduce the spaces $\ell_{p}^{\lambda}(\widehat{F})$ and  $\ell_{\infty}^{\lambda}(\widehat{F})$ derived
by the composition of the two infinite matrices $\Lambda$ and $\widehat{F},$ and show that these spaces are the $BK$-spaces of non-absolute type which are linearly isomorphic to the spaces  $\ell_{p}$ and $\ell_{\infty},$ respectively.

We assume throughout this paper that $\lambda =\left( \lambda_{k} \right)_{k\in\mathbb{N}_0} $ is strictly increasing sequence of positive reals tending to $\infty,$ that is, $0<\lambda_{0}<\lambda_{1}<\cdots$ and $\lambda_{k} \rightarrow \infty,$ as $k\rightarrow \infty.$
\par
The sequence spaces $\ell_{p}^{\lambda}$ and $\ell_{\infty}^{\lambda}$ of non absolute type have been introduced by Mursaleen and Noman (see \cite{M1}) as follows:
\begin{eqnarray*}
\ell_{p}^{\lambda}&:=&\left\lbrace x=(x_{k})\in\omega : \sum_{n=0}^{\infty}\left| \frac{1}{\lambda_{n}}\sum_{k=0}^{n}(\lambda_{k}-\lambda_{k-1})x_{k}\right| ^{p}  < \infty\right\rbrace, \\
\ell_{\infty}^{\lambda}&:=&\left\lbrace x=(x_{k})\in\omega : \sup_{n\in\mathbb{N}}\left| \frac{1}{\lambda_{n}}\sum_{k=0}^{n}(\lambda_{k}-\lambda_{k-1})x_{k}\right| < \infty\right\rbrace.
\end{eqnarray*}
Define the matrix $\Lambda=(\lambda_{nk})$ by
\begin{eqnarray*}
\lambda_{nk}:=\left\{\begin{array}{ccl}
\frac{1}{\lambda_{n}}\left(\lambda_{k}-\lambda_{k-1} \right) &,&0\leq k \leq n, \\
0&,&k>n
\end{array}\right.
\end{eqnarray*}
for all $k,n\in\mathbb{N}_0$ (see \cite{M1,MursaleenNoman}). Then, with the notation of (\ref{matdom}), one can redefine the spaces $\ell_{p}^{\lambda}$ and $\ell_{\infty}^{\lambda}$ as $\ell_{p}^{\lambda}=(\ell_{p})_{\Lambda}$ and $\ell_{\infty}^{\lambda}=(\ell_{\infty})_{\Lambda}.$

Let $n\in\mathbb{N}_0$ and $f_{n}$ be the $nth$ Fibonacci number. The infinite matrix  $\widehat{F}=\left(f_{nk}\right)$ was defined by Kara \cite{Kara} as follows:
\begin{eqnarray*}
f_{nk}:=\left\{\begin{array}{ccl}
-\frac{f_{n+1}}{f_{n}}&,& k=n-1, \\
\frac{f_{n}}{f_{n+1}}&,& k=n, \\
0&,& 0\leq k < n-1~\textrm{ or }~ k>n
\end{array}\right.
\end{eqnarray*}
for all $k,n\in\mathbb{N}_0.$ Define the sequence $y=(y_{n}),$ which will be frequently used, by the $\widehat{F}$-transform of a sequence $x=(x_{n}),$ i.e.,
\begin{eqnarray*}\label{11}
y_{n}=\widehat{F}_{n}(x):=\left\{\begin{array}{ccl}
x_{0}&,&  n=0, \\
\dfrac{_{f_{n}}}{f_{n+1}}x_{n}-\dfrac{_{f_{n+1}}}{f_{n}}x_{n-1}&,&n \geq 1.
\end{array}\right.
\end{eqnarray*}

We employ the technique for obtaining a new sequence space by means of matrix domain. We thus introduce the sequence spaces $\ell_{p}^{\lambda}(\widehat{F})$ and $\ell_{\infty}^{\lambda}(\widehat{F})$ defined, as follows:
\begin{eqnarray*}
\ell_{p}^{\lambda}(\widehat{F})&:=&\left\lbrace x=(x_{k})\in \omega : \sum_{n}\left|  \frac{1}{\lambda_{n}}\sum_{k=0}^{n} \left( \lambda_{k}-\lambda_{k-1} \right) \left( \dfrac{{f_{k}}}{f_{k+1}}x_{k}-\dfrac{{f_{k+1}}}{f_{k}}x_{k-1}    \right) \right|^{p} < \infty  \right\rbrace, \\
\ell_{\infty}^{\lambda}(\widehat{F})&:=&\left\lbrace x=(x_{k})\in \omega : \sup_{n\in\mathbb{N}_0}\left|  \frac{1}{\lambda_{n}}\sum_{k=0}^{n} \left( \lambda_{k}-\lambda_{k-1} \right) \left( \dfrac{{f_{k}}}{f_{k+1}}x_{k}-\dfrac{{f_{k+1}}}{f_{k}}x_{k-1}    \right)\right|  < \infty \right\rbrace.
\end{eqnarray*}
We use the convention that any term with negative subscript is equal to zero, e.g. $\lambda_{-1}=0$ and $x_{-1}=0.$

With the notation of (\ref{matdom}), we can redefine the spaces $\ell_{p}^{\lambda}(\widehat{F})$ and
$\ell_{\infty}^{\lambda}(\widehat{F}),$ as follows:
\begin{eqnarray}\label{c}
\ell_{p}^{\lambda}(\widehat{F})=\left( \ell_{p}^{\lambda}\right)_{\widehat{F}}
~\mbox{ and }~
\ell_{\infty}^{\lambda}(\widehat{F})=\left( \ell_{\infty}^{\lambda}\right)_{\widehat{F}}.
\end{eqnarray}

 It is immediate by (\ref{c}) that the sets $\ell_{p}^{\lambda}(\widehat{F})$ and $\ell_{\infty}^{\lambda}(\widehat{F})$ are linear spaces with coordinatewise addition and scalar multiplication.  On the other hand, we define the matrix $E=\left( e_{nk} \right) $ for all $n,k \in \mathbb{N}_0$ by
\begin{eqnarray}\label{f}
e_{nk}:=\left\{\begin{array}{ccl}
\frac{1}{\lambda_{n}}\left[\left(\lambda_{k}-\lambda_{k-1}\right)\frac{f_{k}}{f_{k+1}}-\left(\lambda_{k+1}-\lambda_{k}\right)\frac{f_{k+2}}{f_{k+1}}\right]&,&k < n, \\
\frac{1}{\lambda_{n}} \left( \lambda_{n}-\lambda_{n-1}\right)\frac{f_{n}}{f_{n+1}}& ,&k=n, \\
0 & ,&k>n.
\end{array}\right.
\end{eqnarray}
Then, it can be easily seen that
\begin{eqnarray*}\label{c0}
E_{n}(x)
= \frac{1}{\lambda_{n}}\sum_{k=0}^{n} \left( \lambda_{k}-\lambda_{k-1} \right) \left( \dfrac{{f_{k}}}{f_{k+1}}x_{k}-\dfrac{{f_{k+1}}}{f_{k}}x_{k-1}    \right)
\end{eqnarray*}
holds for all $n \in \mathbb{N}_0$ and $x=(x_{k})\in\omega$ which leads us to the fact that
\begin{eqnarray}\label{c1}
\ell_{p}^{\lambda}(\widehat{F})=\left( \ell_{p}\right)_{E}
~\mbox{ and }~\ell_{\infty}^{\lambda}(\widehat{F})=\left( \ell_{\infty}\right)_{E}.
\end{eqnarray}
Since $E$ is a triangle, it has a unique inverse
 $E^{-1}=\left( g_{nk} \right) $ for all $n,k \in \mathbb{N}_0$ given by
\begin{eqnarray*}
g_{nk}:=\left\{\begin{array}{ccl}
\lambda_{k}f_{n+1}^{2}\left[\frac{1}{(\lambda_{k}-\lambda_{k-1})f_{k}f_{k+1}}-\frac{1}{(\lambda_{k+1}-\lambda_{k})f_{k+1}f_{k+2}}\right]&,&0 \leq k < n, \\
\frac{\lambda_{n}f_{n+1}^{2}}{(\lambda_{n}-\lambda_{n-1})f_{n}f_{n+1}}&,&k=n, \\
0 & ,&k>n.\end{array}\right.
\end{eqnarray*}

Further, for any sequence $x=(x_{k})$ we define the sequence $y=(y_{k})$ such that $y=Ex,$ that is,
\begin{eqnarray} \label{c2}
y_{k}=E_{k}(x)=\sum_{j=0}^{k-1}
\frac{1}{\lambda_{k}} \left[ \left( \lambda_{j}-\lambda_{j-1} \right) \dfrac{{f_{j}}}{f_{j+1}}- \left( \lambda_{j+1}-\lambda_{j} \right) \dfrac{{f_{j+2}}}{f_{j+1}} \right]x_{j}
+\frac{1}{\lambda_{k}} \left( \lambda_{k}-\lambda_{k-1} \right) \dfrac{{f_{k}}}{f_{k+1}}x_{k}
\end{eqnarray}
for all $k\in\mathbb{N}_0.$

Now, we may begin with the following theorem which is essential in the text.
\begin{thm}\label{T11}
The sequence spaces $\ell^{\lambda}_{p}(\widehat{F})$ and $\ell_{\infty}^{\lambda}(\widehat{F})$ are $BK$-spaces with the norms
\begin{eqnarray*}
\| x \|_{\ell^{\lambda}_{p}(\widehat{F})}&=&\| Ex \|_{p}
=\left( \sum_{n}\left| E_{n}(x)\right|^{p} \right)^{1/p},\\
\| x \|_{\ell^{\lambda}_{\infty}(\widehat{F})}&=&\| Ex \|_{\infty}
=\sup_{n\in\mathbb{N}} \left|E_{n}(x)\right|.
\end{eqnarray*}
\end{thm}
\begin{proof}
Since (\ref{c1}) holds and $\ell_{p}$ and $\ell_{\infty}$ are $BK$-spaces with respect to their natural norms and the matrix $E$ is a triangle, Theorem 4.3.12 of Wilansky \cite{cp6} gives the fact that $\ell^{\lambda}_{p}(\widehat{F})$ and $\ell^{\lambda}_{\infty}(\widehat{F})$ are $BK$-spaces with the given norms.
\end{proof}

\begin{rem}
One can easily check that the absolute property is not satisfied by  $\ell^{\lambda}_{p}(\widehat{F})$ and $\ell_{\infty}^{\lambda} (\widehat{F}),$ that is, $\|x\|_{\ell^{\lambda}_{p}(\widehat{F})}\neq\||x|\|_{\ell^{\lambda}_{p}(\widehat{F})}$ and $\|x\|_{\ell^{\lambda}_{\infty}(\widehat{F})} \neq \||x|\| _{\ell^{\lambda}_{\infty}(\widehat{F})}.$ This shows that
 $\ell^{\lambda}_{p}(\widehat{F})$ and $\ell_{\infty}^{\lambda}(\widehat{F})$  are sequence spaces of non-absolute type, where $|x|=(|x_{k}|).$
\end{rem}

\begin{thm}
The sequence spaces $\ell_{p}^{\lambda}(\widehat{F})$ and $\ell_{\infty}^{\lambda}(\widehat{F})$ of non-absolute type are linearly isomorphic to the spaces $\ell_{p}$ and $\ell_{\infty},$ respectively, that is,
$\ell_{p}^{\lambda}(\widehat{F}) \cong \ell_{p}$ and $\ell_{\infty}^{\lambda}(\widehat{F}) \cong \ell_{\infty}.$
\end{thm}
\begin{proof}
To prove the fact $\ell_{p}^{\lambda}(\widehat{F}) \cong \ell_{p},$ we should show the existence of a linear bijection between the spaces  $\ell_{p}^{\lambda}(\widehat{F})$ and $\ell_{p}.$ Consider the transformation $T$ defined, with the notation of (\ref{c2}), from $\ell_{p}^{\lambda}(\widehat{F})$ to $\ell_{p}$ by
$Tx =y=Ex \in \ell_{p}$ for every $x \in \ell_{p}^{\lambda}(\widehat{F}).$ Since $T$ has a matrix representation, the linearity of $T$ is clear. Further, it is trivial that $x=\theta$ whenever $Tx=\theta.$ Hence, $T$ is injective.

Further, let $y \in (y_{k})\in \ell_{p}.$ Now, we define the sequence $x=(x_{k})$ by
\begin{eqnarray*}\label{c3}
x_{k}=\sum_{j=0}^{k}\sum_{i=j-1}^{j} (-1)^{j-i}
\frac{f_{k+1}^{2}\lambda_{i}y_{i}}{(\lambda_{j}-\lambda_{j-1})f_{j}f_{j+1}}
\end{eqnarray*}
for all $k\in\mathbb{N}_0.$ It is immediate by the fact that $Ex=y\in\ell_{p}$ that $x\in \ell_{p}^{\lambda}(\widehat{F}).$ Hence, $T$ is surjective.

Moreover, for every $x \in \ell_{p}^{\lambda}(\widehat{F})$ we have $\|Tx\|_{p}=\|y\|_{p}=\|Ex\|_{p}=\|x\|_{\ell^{\lambda}_{p}{(\widehat{F})}}$ which means that $T$ is norm preserving. Consequently, $T$ is a linear bijection which shows that $ \ell_{p}^{\lambda}(\widehat{F})$ and $\ell_{p}$ are linearly isomorphic.

Similarly, one can show that $\ell_{\infty}^{\lambda}(\widehat{F})\cong\ell_{\infty}.$ So, we omit the details.

This concludes the proof.
\end{proof}

\begin{thm}
Except the case $p=2,$ the space $\ell_{p}^{\lambda}(\widehat{F}) $ is not an inner product space and hence is not a Hilbert space.
\end{thm}
\begin{proof}
We have to prove that the space $\ell_{2}^{\lambda}(\widehat{F})$ is the only Hilbert space among the
$\ell_{p}^{\lambda}(\widehat{F})$ spaces. Since the space $\ell_{2}^{\lambda}(\widehat{F})$
is the $BK$-space with the norm $\| x \| _{\ell_{2}^{\lambda} (\widehat{F}) }=\| Ex\|_{2}$ by Theorem \ref{T11} and its norm can be obtained from an inner product, i.e., the equality
\[\| x \| _{\ell_{2}^{\lambda} (\widehat{F}) }= \left\langle x, x\right\rangle ^{1/2}
=\left\langle Ex, Ex\right\rangle ^{1/2}_{2}\]
holds for all $x \in \ell_{2}^{\lambda} (\widehat{F}),$ the space $\ell_{2}^{\lambda} (\widehat{F})$ is a Hilbert space; where $\left\langle \cdot,\cdot\right\rangle_{2} $ denotes the inner product on $\ell_{2}.$

Let us consider the sequences $u=(u_{k})$ and $v=(v_{k})$ defined by
\begin{eqnarray*}\label{eq1}
u_{k}&:=&\left\{\begin{array}{ccl}
1&,& k=0, \\
f_{2}^{2}+f_{2}&,&k=1 \\
f_{3}^{2}\left( 1+ \frac{1}{f_{2}}\right)-\frac{\lambda_{1}f_{3}}{(\lambda_{2}-\lambda_{1})f_{2}}& ,&k \geq 2,
\end{array}\right.\\
\label{eq2}
v_{k}&:=&\left\{\begin{array}{ccl}
1&,& k=0, \\
f_{2}^{2}-\left(\frac{\lambda_{1}+\lambda_{0}}{\lambda_{1}-\lambda_{0}}\right)f_{2}&,&k=1, \\
f_{3}^{2}\left[1-\frac{\lambda_{1}+\lambda_{0}}{(\lambda_{1}-\lambda_{0})f_{2}}\right]+\frac{\lambda_{1}f_{3}}{(\lambda_{2}-\lambda_{1})f_{2}}&,&k \geq 2.
\end{array}\right.
\end{eqnarray*}
Thus, we have $Eu=\left(1,1,0,0,\ldots\right)$ and
$Ev=\left(1,-1,0,0,\ldots\right).$ Therefore, it can be easily seen with $p\neq 2$ that
\begin{eqnarray*}
\| u+v \|^{2}_{\ell_{p}^{\lambda}(\widehat{F})}+ \| u-v \|^{2}_{\ell_{p}^{\lambda}(\widehat{F})}=8
\neq 4\left( 2^{2/p}\right) =2\left( \| u \|^{2}_{\ell_{p}^{\lambda}(\widehat{F})}
+\| v \|^{2}_{\ell_{p}^{\lambda}(\widehat{F})}\right),
\end{eqnarray*}
i.e., the norm $\|\cdot\|_{\ell^{\lambda}_{p}{(\widehat{F})}}$ with $p\neq 2$ doesn't satisfy the parallelogram identity. This means that the norm $\|\cdot\|_{\ell_{p}^{\lambda}(\widehat{F})}$ can't be obtained from an inner product. Hence,
$\ell_{p}^{\lambda}(\widehat{F}) $ with $p \neq 2$ is not a Hilbert space.
\end{proof}
\begin{rem}
$\ell_{\infty}^{\lambda}(\widehat{F})$ is not an Hilbert space.
\end{rem}
\begin{thm}
If $1 \leq p < q < \infty,$ then the inclusion
$\ell_{p}^{\lambda}(\widehat{F}) \subset \ell_{q}^{\lambda}(\widehat{F})$ strictly holds.
\end{thm}
\begin{proof}
Let $1 \leq p < q < \infty.$ Since $\ell_{p} \subset \ell_{q}$, we have $\ell_{p}^{\lambda}(\widehat{F}) \subset \ell_{q}^{\lambda}(\widehat{F}).$ Further, since the inclusion $\ell_{p} \subset \ell_{q}$ is strict, there exists a sequence $x=(x_{k} )\in \ell_{q}$ but not in $\ell_{p}.$ Let us now define the sequence $y=(y_{k})$ in terms of the sequence $x=(x_{i}),$ as follows:
\begin{eqnarray*}
y_{k}=\sum_{j=0}^{k}\sum_{i=j-1}^{j} (-1)^{j-i}
\frac{\lambda_{i}f_{k+1}^{2}}{(\lambda_{j}-\lambda_{j-1})f_{j}f_{j+1}}x_{i}
\end{eqnarray*}
for all $k\in\mathbb{N}_0.$ Then, we have for all $n\in\mathbb{N}_0$ that
\[ E_{n}(y)=\frac{1}{\lambda_{n}}\sum_{k=0}^{n}\left( \lambda_{k}-\lambda_{k-1}\right)\left( \frac{f_{k}}{f_{k+1}}x_{k} -\frac{f_{k+1}}{f_{k}}x_{k-1}\right)=x_{n}  \]
which shows that $Ey=x\in\ell_{q} \backslash \ell_{p}.$ Hence, $y \in \ell_{q}^{\lambda}(\widehat{F})$ but is not in $\ell_{p}^{\lambda}(\widehat{F}).$
That is to say that the inclusion $\ell_{p}^{\lambda}(\widehat{F})  \subset \ell_{q}^{\lambda}(\widehat{F})$ is strict. This concludes the proof.
\end{proof}
\begin{thm}
The inclusions
$\ell_{p}^{\lambda}(\widehat{F})  \subset
c_{0}^{\lambda}(\widehat{F})  \subset
c^{\lambda}(\widehat{F})  \subset
\ell_{\infty}^{\lambda}(\widehat{F})
$
strictly hold.
\end{thm}
\begin{proof}
It is trivial that the inclusion $c_{0}^{\lambda}(\widehat{F})\subset
c^{\lambda}(\widehat{F})$ strictly holds. Let $x=(x_{k})\in \ell_{p}^{\lambda}(\widehat{F}).$ This means that $Ex \in \ell_{p}.$ Since $\ell_{p}\subset c_{0},$  $Ex\in c_{0}$ which gives $x\in c_{0}^{\lambda}(\widehat{F}).$ Hence, $\ell_{p}^{\lambda}(\widehat{F})\subset
c_{0}^{\lambda}(\widehat{F})$ holds. Now, we have to show that the inclusion is strict.

Let us define the sequence $x=(x_{k})$ by
\begin{eqnarray*}
x_{k}=\sum_{j=0}^{k}\sum_{i=j-1}^{j} (-1)^{j-i}
\frac{\lambda_{i}f_{k+1}^{2}}{\left(\lambda_{j}-\lambda_{j-1}\right)\left( i+1\right)^{1/p}f_{j}f_{j+1}}
\end{eqnarray*}
for all $k \in \mathbb{N}_0.$ Then for all $n \in \mathbb{N}_0$ we have 
$E_{n}(x)=(n+1)^{-1/p}$ which shows that $Ex$ is not in $\ell_{p}$ but is in $c_{0}.$ Thus, the sequence $x$ belongs to the set $c_{0}^{\lambda}(\widehat{F})\setminus\ell_{p}^{\lambda}(\widehat{F}).$ Hence the inclusion $\ell_{p}^{\lambda}(\widehat{F})\subset
c_{0}^{\lambda}(\widehat{F})$ is strict.

Since $c \subset \ell_{\infty}$ holds, we have
$c^{\lambda}(\widehat{F})  \subset
\ell_{\infty}^{\lambda}(\widehat{F}).$ Let us consider the
sequence $y=(y_{k})$ defined by
\begin{eqnarray*}
y_{k}=\sum_{j=0}^{k}\sum_{i=j-1}^{j} (-1)^{j}\frac{\lambda_{i}f_{k+1}^{2}}{(\lambda_{j}-\lambda_{j-1})f_{j}f_{j+1}}
\end{eqnarray*}
for all $k \in \mathbb{N}_0.$ Then for all $n\in\mathbb{N}_0$ we have  $E_{n}(y)=(-1)^{n}$ which shows that
$Ey\in \ell_{\infty}\backslash c.$ Thus, $y \in \ell_{\infty}^{\lambda}(\widehat{F}) \backslash
  c^{\lambda}(\widehat{F}).$ Therefore, the inclusion
  $c^{\lambda}(\widehat{F})\subset\ell_{\infty}^{\lambda}(\widehat{F})$ is strict. This completes the proof.
\end{proof}

\begin{thm}
The inclusion $\ell_{\infty} \subset \ell_{\infty}^{\lambda}(\widehat{F}) $  strictly holds.
\end{thm}
\begin{proof}
Let $x=(x_{k})\in \ell_{\infty}.$ Then, we have
\begin{eqnarray*}
\| x \|_{\ell_{\infty}^{\lambda}(\widehat{F})}=\sup_{n \in \mathbb{N}_0}\left|  E_{n}(x)\right|
= \sup_{n \in \mathbb{N}_0}\left| \frac{1}{\lambda_{n}}\sum_{k=0}^{n} \left( \lambda_{k}-\lambda_{k-1} \right) \left(\dfrac{{f_{k}}}{f_{k+1}}x_{k}-\dfrac{{f_{k+1}}} {f_{k}}x_{k-1}\right)\right|
\end{eqnarray*}
Since $\frac{f_{k}}{f_{k+1}}\leq 1$ and $\frac{f_{k+1}}{f_{k}}\leq 2$ for all $k \in \mathbb{N}_0,$ one can see that
\[ \| x \|_{\ell_{\infty}^{\lambda}(\widehat{F})}\leq 4 \| x \|_{\infty}.
\sup_{n \in \mathbb{N}_0}\left| \frac{1}{\lambda_{n}}\sum_{k=0}^{n} \left( \lambda_{k}-\lambda_{k-1} \right) \right|
\leq  4 \| x \|_{\infty}.\]
Hence, the inclusion $\ell_{\infty} \subset \ell_{\infty}^{\lambda}(\widehat{F})$ holds.

Now, let us consider $t=(t_{k})$ defined by
\begin{eqnarray*}\label{eq3}
t_{k}:= \left\{\begin{array}{ccl}
1&,& k=0, \\
f_{k+1}^{2}\left(\sum_{j=1}^{k}\frac{1}{f_{j}f_{j+1}}+1\right)&,&k\geq1.
\end{array}\right.
\end{eqnarray*}
Then, we obtain $E_{n}(t)=1$ for all $n\in\mathbb{N}_0$ which leads to the fact that $t\in \ell_{\infty}^{\lambda}(\widehat{F})\backslash \ell_{\infty}.$ Hence, the inclusion
$\ell_{\infty} \subset \ell_{\infty}^{\lambda}(\widehat{F}) $
is strict.

This completes the proof.
\end{proof}
\begin{thm}
If the inclusion $\ell_{p} \subset \ell_{p}^{\lambda}(\widehat{F}) $ holds, then $\left( 1/\lambda_{n}\right)\in \ell_{p}.$
\end{thm}
\begin{proof}
Let us assume that the inclusion $\ell_{p} \subset \ell_{p}^{\lambda}(\widehat{F})$ holds and consider the sequence $e^{(0)}=\left\lbrace 1,0,0,0,\ldots\right\rbrace \in \ell_{p}.$ Then, we have $e^{(0)}\in\ell_{p}^{\lambda}(\widehat{F})$ by our assumption and hence $Ee^{(0)}\in\ell_{p}.$ We have
$E_{n}\left(e^{(0)}\right)=\left(3\lambda_{0}-2\lambda_{1}\right)/\lambda_{n}$ and therefore we obtain
\[\sum_{n}\left| E_{n}\left( e^{(0)}\right)\right|^{p}
=\left|3\lambda_{0}-2\lambda_{1} \right|^{p} \sum_{n}\left( \frac{1}{\lambda_{n}}\right)^{p} < \infty \]
which implies that $\left(1/\lambda_{n}\right)\in\ell_{p}.$

This completes the proof.
\end{proof}

\begin{lem} \cite{M1} 
\label{l11}
If $\left( 1/\lambda_{n}\right)\in \ell_{1},$ then $M=\sup\limits_{n \in \mathbb{N}_0}\sum\limits_{n=k}^{\infty}\frac{\lambda_{k}-\lambda_{k-1}}{\lambda_{n}}<\infty.$
\end{lem}
\begin{thm}
If  $\left(1/\lambda_{n}\right)\in \ell_{1},$ then the inclusion $\ell_{p}\subset\ell_{p}^{\lambda}(\widehat{F})$ strictly holds.
\end{thm}
\begin{proof}
Let $x=(x_{k})\in \ell_{p}$ with $p>1.$ Then, by applying H\"{o}lder's inequality we have
\begin{eqnarray*}
\left|E_{n}(x)\right|&=&\left| \sum_{k=0}^{n}\left( \frac{\lambda_{k}-\lambda_{k-1}}{\lambda_{n}}\right)
\left( \frac{f_{k}}{f_{k-1}}x_{k}-\frac{f_{k+1}}{f_{k}}x_{k-1}\right)\right| \\
&\leq&\left[  \sum_{k=0}^{n}\left( \frac{\lambda_{k}-\lambda_{k-1}}{\lambda_{n}}\right)
 \left|\frac{f_{k}}{f_{k-1}}x_{k}-\frac{f_{k+1}}{f_{k}}x_{k-1} \right|^{p} \right]^{1/p}\left[\sum_{k=0}^{n}\left(\frac{\lambda_{k}-\lambda_{k-1}}{\lambda_{n}} \right)  \right] ^{1-1/p}
\end{eqnarray*}
which gives
\[ \left| E_{n}(x)\right|^{p}\leq \sum_{k=0}^{n}\left( \frac{\lambda_{k}-\lambda_{k-1}}{\lambda_{n}}\right)
 \left|\frac{f_{k}}{f_{k-1}}x_{k}-\frac{f_{k+1}}{f_{k}}x_{k-1} \right|^{p}.
\]
By Lemma \ref{l11} we have
\[ \left| E_{n}(x)\right|^{p}
\leq \sum_{k=0}^{n}M
 \left|\frac{f_{k}}{f_{k-1}}x_{k}-\frac{f_{k+1}}{f_{k}}x_{k-1} \right|^{p}
\]
and hence,
\[\| x \|^{p}_{\ell_{p}^{\lambda}(\widehat{F}) } \leq
M2^{2p-1}\left( \sum_{k} \left| x_{k}\right|^{p}
+ \sum_{k}\left| x_{k-1}\right|^{p}\right)
\leq M2^{2p}\| x \|^{p}_{\ell_{p}^{\lambda}}.
\]
Therefore, $\|x\|_{\ell_{p}^{\lambda}(\widehat{F}) }
\leq 4M^{1/p}\|x\|_{p}<\infty.$ This shows that $x\in\ell_{p}^{\lambda}(\widehat{F}).$ Hence the inclusion $\ell_{p}\subset\ell_{p}^{\lambda}(\widehat{F})$ holds for $p>1.$

Now, let us consider the sequence $v=\left( v_{k}\right) $ defined by
\begin{eqnarray*}\label{eq4}
v_{k}:=\left\{\begin{array}{ccl}
1&,& k=0, \\
f_{2}^{2}- \frac{\lambda_{0}f_{2}}{(\lambda_{1}-\lambda_{0})f_{1}} &,&k\geq1.
\end{array}\right.
\end{eqnarray*}
Then, we have $Ev=e^{(0)}\in \ell_{p}.$ Therefore, $v\in \ell_{p}^{\lambda}(\widehat{F})\setminus\ell_{p}$.
This means that the inclusion $\ell_{p}\subset \ell_{p}^{\lambda}(\widehat{F}) $ is strict.

Similarly, one can show that the inclusion $\ell_{1}\subset\ell_{1}^{\lambda}(\widehat{F}) $ also strictly holds.
\end{proof}

It is known from Theorem 2.3 of Jarrah and Malkowsky \cite{ajm} that if $T$ is a triangle then the domain
$\lambda_T$ of $T$ in a normed sequence space $\lambda$ has a basis if and only if
$\lambda$ has a basis. As a direct consequence of this fact, since  the transformation $T$ defined from $\ell^{\lambda}_{p}(\widehat{F})$ to $\ell_{p},$ is an isomorphism, the inverse image of the basis $\left\lbrace e^{(k)}\right\rbrace_{k\in\mathbb{N}_0}$ of the space $\ell_{p}$ is the basis for the new space $\ell^{\lambda}_{p}(\widehat{F})$ with $1\leq p<\infty,$ we have:
\begin{cor}
Define the sequence $b^{(k)}=\big\lbrace b_{n}^{(k)}\big\rbrace_{n\in\mathbb{N}_0}$ for every fixed $k\in\mathbb{N}_0$ by
\begin{eqnarray*}\label{c5}
b_{n}^{(k)}:=\left\{\begin{array}{ccl}
0&,&n < k, \\
\frac{\lambda_{k}f_{n+1}^{2}}{(\lambda_{k}-\lambda_{k-1})f_{k}f_{k+1}}& ,&n=k, \\
\frac{\lambda_{k}f_{n+1}^{2}}{(\lambda_{k}-\lambda_{k-1})f_{k}f_{k+1}}-\frac{\lambda_{k}f_{n+1}^{2}}{(\lambda_{k+1}-\lambda_{k})f_{k+1}f_{k+2}}&,&n>k.
\end{array}\right.
\end{eqnarray*}
Then, the following statements hold:
\begin{enumerate}
\item[(i)] The space $\ell^{\lambda}_{\infty}(\widehat{F})$ has no Schauder basis.
\item[(ii)] The sequence $\big\{b^{(k)}\big\}_{k\in\mathbb{N}_0}$ is a basis for the space $\ell^{\lambda}_{p}(\widehat{F})$ and every $x\in\ell^{\lambda}_{p}(\widehat{F})$ has a unique representation of the form $x=\sum\limits_{k}\alpha_{k}b^{(k)},$ where $\alpha_{k}=E_{k}(x)$ for all $k\in\mathbb{N}_0.$
\end{enumerate}
\end{cor}
\begin{cor}
While the space $\ell_{p}^{\lambda}(\widehat{F})$ is separable but $\ell_{\infty}^{\lambda}(\widehat{F})$ is not separable.
\end{cor}
\section{ The $\alpha$-, $\beta$- and $\gamma$-duals of the spaces $\ell_{p}^{\lambda}(\widehat{F})$ and $\ell_{\infty}^{\lambda}(\widehat{F})$}
In this section, we determine the $\alpha$-,$\beta$- and $\gamma$-duals of the sequence spaces $\ell_{p}^{\lambda}(\widehat{F})$ and $\ell_{\infty}^{\lambda}(\widehat{F})$ of non-absolute type.

We assume throughout that the sequences $x=(x_{k})$ and $y=(y_{k})$ are connected by the relation (\ref{c2}). Let $A=(a_{nk})$ be an infinite matrix. Now, we may begin with quoting the following lemmas which are required for proving the next theorems.
\begin{lem} \cite{Michael}
$A=(a_{nk})  \in \left(\ell_{p}:\ell_{1}\right)$ if and only if
\begin{enumerate}
\item[(i)] For $1< p \leq \infty,$
\begin{eqnarray*}\label{eq5}
\sup_{K \in \mathcal{F}} \sum \limits_{k}\left| \sum_{n\in K}a_{nk}\right|^{q}< \infty.
\end{eqnarray*}
\item[(ii)] For $p=1,$
\begin{eqnarray*}\label{eq6}
\sup_{k \in \mathbb{N}_0}\sum_{n}\left| a_{nk}\right|<\infty.
\end{eqnarray*}
\end{enumerate}
\label{l1}
\end{lem}
\begin{lem}\cite{Michael} The following statements hold:
\begin{enumerate}
\item[(i)] Let $1 < p <\infty.$ Then,
$A=(a_{nk})  \in \left(\ell_{p}:c \right)$ if and only if
\begin{eqnarray}\label{eq7}
&&\lim_{n\to\infty}a_{nk} \mbox{~exists~for~each~fixed~} k \in \mathbb{N}_0,
\\
&&\sup_{n\in\mathbb{N}_0}\sum_{k} \left| a_{nk}\right|^{q}<\infty.\label{eq8}
\end{eqnarray}
\item[(ii)] $A=(a_{nk})  \in \left(\ell_{1}:c\right)$ if and only if (\ref{eq7}) holds and
\begin{eqnarray}
\sup_{n,k \in \mathbb{N}_0}\left| a_{nk}\right|<\infty.
\label{eq9}
\end{eqnarray}
\item[(iii)] $A=(a_{nk})\in\left(\ell_{\infty}:c\right)$ if and only if
\begin{eqnarray*}\label{eq10}
&&\sup_{n\in\mathbb{N}_0}\sum_{k}\left|a_{nk}\right|<\infty,\\
&&\lim_{n\to\infty}\sum_{k}\left|a_{nk}-\lim_{n\to\infty}a_{nk}\right|=0.
\label{eq11}
\end{eqnarray*}
\end{enumerate}
\label{l2}
\end{lem}
\begin{lem}\cite{Michael}
Let $1< p \leq \infty.$ Then, the following statements hold:
\begin{enumerate}
\item[(i)] $A=(a_{nk}) \in \left( \ell_{p}:\ell_{\infty} \right)$ if and only if (\ref{eq8}) holds.
\item[(ii)] $A=(a_{nk}) \in \left( \ell_{1}:\ell_{\infty} \right)$ if and only if (\ref{eq9}) holds.
\end{enumerate}
\label{l3}
\end{lem}

\begin{thm}
Define the sets $d_{1}$ and $d_{2}$ by
\begin{eqnarray*}
d_{1}&:=&\left\lbrace a=(a_{k}) \in \omega :  \sup_{K \in \mathcal{F}} \sum \limits_{k}\left| \sum_{n \in K} b_{nk}\right|^{q} < \infty
 \right\rbrace, \\
d_{2} &:=&\left\lbrace a=(a_{k}) \in \omega :  \sup_{k \in \mathbb{N}_0} \sum \limits_{n}\left| b_{nk}\right| < \infty
  \right\rbrace,
\end{eqnarray*}
where the matrix $B=(b_{nk})$ is defined via the sequence $a=(a_{n})\in\omega$ by
\begin{eqnarray*}
b_{nk}:=\left\{\begin{array}{ccl}
\left[\frac{\lambda_{k}f_{n+1}^{2}}{(\lambda_{k}-\lambda_{k-1})f_{k}f_{k+1}}-\frac{\lambda_{k}f_{n+1}^{2}}{(\lambda_{k+1}-\lambda_{k})f_{k+1}f_{k+2}}\right] a_{n}&,&k < n, \\
\frac{\lambda_{k}f_{n+1}^{2}}{(\lambda_{k}-\lambda_{k-1})f_{k}f_{k+1}}&,&k=n, \\
0&,&k>n
\end{array}\right.
\end{eqnarray*}
for all $n,k\in\mathbb{N}_0.$ Then $\big[\ell_{1}^{\lambda}(\widehat{F})\big]^{\alpha}=d_{2}$ and $\big[\ell_{p}^{\lambda}(\widehat{F})\big]^{\alpha}=d_{1}.$
\end{thm}
\begin{proof}
Let $a=(a_{n})\in\omega.$ Then, we immediately derive by (\ref{c2}) that
\begin{eqnarray}\label{c16}
a_{n}x_{n}=\sum_{k=0}^{n}\sum_{j=k-1}^{k}(-1)^{k-j}\frac{\lambda_{j}f_{n+1}^{2}}{(\lambda_{k}-\lambda_{k-1})f_{k}f_{k+1}}a_{n}y_{j}=B_{n}(y)
\end{eqnarray}
for all $n\in\mathbb{N}_0.$ Thus, we observe by (\ref{c16}) that $ax=\left( a_{n}x_{n}\right)\in\ell_{1}$ when $x=(x_{k})\in \ell_{p}^{\lambda}(\widehat{F}) $ if and only if $By\in\ell_{1}$ when $y=(y_{k})\in\ell_{p},$ i.e., $a=(a_{n})$ is in the $\alpha$-dual of the space $\ell_{p}^{\lambda}(\widehat{F})$ if and only if $B \in\left(\ell_{p}:\ell_{1}\right).$ Therefore, we see by Lemma \ref{l1} that
$a \in\big[\ell_{p}^{\lambda}(\widehat{F})\big]^{\alpha} $ iff
$$ \sup_{K \in \mathcal{F}} \sum_{k}\left| \sum_{n \in
K}b_{nk}\right|^{q}< \infty$$ which gives that
$\big[\ell_{p}^{\lambda}(\widehat{F})\big]^{\alpha}=d_{1}.$

Similarly, we get from (\ref{c16}) that
$a \in\big[\ell_{1}^{\lambda}(\widehat{F})\big]^{\alpha} $ if and only if $B \in \left( \ell_{1}: \ell_{1}\right)$ which is equivalent to
\[\sup_{k \in \mathbb{N}_0} \sum \limits_{n}\left| b_{nk}\right| < \infty .\]
This leads us to the desired result that $\big[\ell_{1}^{\lambda}(\widehat{F})\big]^{\alpha}=d_{2}.$
\end{proof}

\begin{thm}\label{T1}
Define the sets $d_{3},d_{4},d_{5},d_{6},d_{7}$ and $d_{8}$ by
\begin{eqnarray*}
d_{3}&:=&\left\lbrace a=(a_{k}) \in \omega :  \sum_{j=k+1}^{\infty}  a_{j} f_{j+1}^{2} \mbox{~exists~for~each~} k\in \mathbb{N}_0   \right\rbrace,\\
d_{4}&:=&\left\lbrace a=(a_{k}) \in \omega : \sup_{n\in \mathbb{N}_0} \sum_{k=0}^{n-1}\left| \bar{a}_{k}(n)\right|^{q}< \infty \right\rbrace , \\
d_{5}&:=&\left\lbrace a=(a_{k}) \in \omega : \sup_{n \in \mathbb{N}_0}\left| \frac{\lambda_{n}f^{2}_{n+1}}{(\lambda_{n}-\lambda_{n-1})f_{n}f_{n+1}}{a}_{n}\right|< \infty \right\rbrace , \\
d_{6}&:=&\left\lbrace a=(a_{k}) \in \omega : \sup_{n,k\in \mathbb{N}_0} \left| \bar{a}_{k}(n)\right|< \infty \right\rbrace , \\
d_{7}&:=&\left\lbrace a=(a_{k}) \in \omega :
\lim\limits_{n \rightarrow \infty} \sum_{k}\left| \bar{a}_{k}(n)-\bar{a}_{k}\right|=0 \right\rbrace,\\
d_{8}&:=&\left\lbrace a=(a_{k}) \in \omega : \sup_{n\in \mathbb{N}_0} \sum_{k}\left| \bar{a}_{k}(n)\right|< \infty \right\rbrace,
\end{eqnarray*}
where
\begin{eqnarray*}
\bar{a}_{k}(n)=\lambda_{k}\left\{\frac{a_{k}f_{k+1}^{2}}{(\lambda_{k}-\lambda_{k-1})f_{k}f_{k+1}}+
 \left[\frac{1}{(\lambda_{k}-\lambda_{k-1})f_{k}f_{k+1}}- \frac{1}{(\lambda_{k+1}-\lambda_{k})f_{k+1}f_{k+2}}\right]\sum_{j=k+1}^{n}f_{j+1}^{2}a_{j}
 \right\},~~k<n
\end{eqnarray*}
and $\bar{a}_{k}=\lim\limits_{n\rightarrow \infty} \bar{a}_{k}(n).$ Then $\big[\ell_{p}^{\lambda}(\widehat{F})\big]^{\beta}=d_{3}\cap d_{4}\cap d_{5},
\big[\ell_{1}^{\lambda}(\widehat{F})\big]^{\beta}=d_{3}\cap d_{5}\cap d_{6}~\mbox{and}~
 \big[\ell_{\infty}^{\lambda}(\widehat{F})\big]^{\beta}=d_{4}\cap d_{7}\cap d_{8},$ where $1<p<\infty.$
\end{thm}
\begin{proof}
Let $a=(a_{k})\in \omega$ and consider the equality
\begin{eqnarray*}
\sum_{k=0}^{n} a_{k}x_{k}&=&\sum_{k=0}^{n} \left\lbrace \sum_{j=0}^{k}
\left[\sum_{i=j-1}^{j} (-1)^{j-i}
\frac{\lambda_{i}f_{k+1}^{2}}{(\lambda_{j}-\lambda_{j-1})f_{j}f_{j+1}}y_{i}\right]\right\rbrace a_{k} \\
&=&\sum_{k=0}^{n-1} \lambda_{k}\left\{\frac{a_{k}f_{k+1}^{2}}{(\lambda_{k}-\lambda_{k-1})f_{k}f_{k+1}}+\left[ \frac{1}{(\lambda_{k}-\lambda_{k-1})f_{k}f_{k+1}}-
\frac{1}{(\lambda_{k+1}-\lambda_{k})f_{k+1}f_{k+2}}\right]\sum_{j=k+1}^{n}f_{j+1}^{2}a_{j} \right\}y_{k}\\ && +
\frac{\lambda_{n}f_{n+1}^{2}}{(\lambda_{n}-\lambda_{n-1})f_{n}f_{n+1}}a_{n}y_{n} \\
&=&\sum_{k=0}^{n-1}\bar{a}_{k}(n)y_{k}+
\frac{\lambda_{n}f_{n+1}^{2}}{(\lambda_{n}-\lambda_{n-1})f_{n}f_{n+1}}a_{n}y_{n}\\
&=&T_{n}(y) \mbox{~for~ all~} n\in\mathbb{N}_0,
\end{eqnarray*}
 where $T=\left(t_{nk}\right) $ is defined by
\begin{eqnarray*}
t_{nk}:=\left\{\begin{array}{ccl}
\bar{a}_{k}(n)&,&k < n, \\
\frac{\lambda_{n}f_{n+1}^{2}}{(\lambda_{n}-\lambda_{n-1})f_{n}f_{n+1}}&,&k=n, \\
0&,&k>n
\end{array}\right.
\end{eqnarray*}
for all $n,k \in \mathbb{N}_0.$ Then we have $ax=(a_{k}x_{k})\in cs$ whenever $x=(x_{k}) \in \ell_{p}^{\lambda}(\widehat{F})$ if and only if $Ty \in c$ whenever $y=(y_{k}) \in \ell_{p}.$ Therefore $a=(a_{k})\in\big[\ell_{p}^{\lambda}(\widehat{F})\big]^{\beta}$ if and only if $T \in \left( \ell_{p} : c\right)$ with $1\leq p \leq \infty.$ Then, we derive by using Lemma \ref{l2} for $1<p<\infty$ that
\begin{eqnarray*}\label{c17}
&&\sum_{j=k+1}^{\infty} a_{j}f_{j+1}^{2} ~\textrm{ exists for each }~k\in\mathbb{N}_0,\\
\label{c18}
&&\sup_{n\in\mathbb{N}_0}\sum_{k=0}^{n-1}\left| \bar{a}_{k}(n)\right|^{q} < \infty,\\
\label{c19}
&&\sup_{n\in\mathbb{N}_0} \left| \frac{\lambda_{n}f_{n+1}^{2}}{(\lambda_{n}-\lambda_{n-1})f_{n}f_{{n+1}}}a_{n}\right| < \infty.
\end{eqnarray*}
Therefore we conclude that $\big[\ell_{p}^{\lambda}(\widehat{F})\big]^{\beta}
=d_{3}\cap d_{4}\cap d_{5} ~\mbox{for}~ 1<p<\infty. $  \\
Similarly, for $p=1~\mbox{and}~ p=\infty$ we can see by using Parts (ii) and (iii) of Lemma \ref{l2} that $\big[\ell_{1}^{\lambda}(\widehat{F})\big]^{\beta}=d_{3}\cap d_{5}\cap d_{6}$ and $\big[\ell_{\infty}^{\lambda}(\widehat{F})\big]^{\beta}=d_{4}\cap d_{7}\cap d_{8}.$ This completes the proof.
\end{proof}

\begin{thm}
Let $1< p \leq \infty.$ Then we have: $\big[\ell_{1}^{\lambda}(\widehat{F})\big]^{\gamma}= d_{5}\cap d_{6}$ and $\big[\ell_{p}^{\lambda}(\widehat{F})\big]^{\gamma}=d_{5}\cap d_{8}.$
\end{thm}
\begin{proof}
This is obtained in the similar way used in the proof of Theorem \ref{T1} with Lemma \ref{l3} instead of Lemma \ref{l2}.
\end{proof}
\section{ Some matrix transformations on the sequence spaces $ \ell_{p}^{\lambda}(\widehat{F})$ and
$ \ell_{\infty}^{\lambda}(\widehat{F})$}

In this section, we characterize the classes
$ \left( \ell_{p}^{\lambda}(\widehat{F}): \ell_{\infty} \right),$
$ \left( \ell_{p}^{\lambda}(\widehat{F}): c_{0} \right),$
$ \left( \ell_{p}^{\lambda}(\widehat{F}): c \right),$
$ \left( \ell_{p}^{\lambda}(\widehat{F}): \ell_{1} \right),$
$ \left( \ell_{1}^{\lambda}(\widehat{F}): \ell_{p} \right),$
 and
$ \left( \ell_{\infty}^{\lambda}(\widehat{F}): \ell_{p} \right)$ of matrix transformations
 where $1 \leq p \leq \infty.$

We assume that the sequences $x$ and $y$ are connected by $y=Ex.$ We write for simplicity in notation that
\begin{eqnarray*}
e_{nk}(m)=\lambda_{k}\left\{\frac{f_{k+1}^{2}a_{nk}}{(\lambda_{k}-\lambda_{k-1})f_{k}f_{k+1}}+
  \left[\frac{1}{(\lambda_{k}-\lambda_{k-1})f_{k}f_{k+1}}-
  \frac{1}{(\lambda_{k+1}-\lambda_{k})f_{k+1}f_{k+2}}\right]\sum_{j=k+1}^{m}f_{j+1}^{2}a_{nj}\right\},
\end{eqnarray*}
where  $k<m$ and
\begin{eqnarray}\label{eq400}
\quad e_{nk}=\lambda_{k}\left\{\frac{f_{k+1}^{2}a_{nk}}{(\lambda_{k}-\lambda_{k-1})f_{k}f_{k+1}}+
\left[\frac{1}{(\lambda_{k}-\lambda_{k-1})f_{k}f_{k+1}}-\frac{1}{(\lambda_{k+1}-\lambda_{k})f_{k+1}f_{k+2}}\right]\sum_{j=k+1}^{\infty}f_{j+1}^{2}a_{nj}\right\}
\end{eqnarray}
for all $k,m,n\in\mathbb{N}_0$ provided the convergence of the series.

Now, we quote the following lemmas which are needed in proving our theorems:
\begin{lem} \cite{Michael} \label{l6}
$A=(a_{nk})\in \left(\ell_{p} : c_{0}\right)$ if and only if
\begin{enumerate}
\item[(i)] For $p=1,$
\begin{eqnarray}\label{eq12}
&&\lim_{n\to\infty}a_{nk}=0 ~\mbox{for~all~} k \in \mathbb{N}_0,\\
\label{eq13}
&&\sup_{n,k \in \mathbb{N}_0}\left| a_{nk}\right|< \infty.\nonumber
\end{eqnarray}
\item[(ii)] For $1< p < \infty,$ (\ref{eq12}) holds and
\begin{eqnarray*}\label{eq14}
\sup_{n\in\mathbb{N}_0}\sum_{k}\left| a_{nk}\right|^{q}< \infty.
\end{eqnarray*}
\item[(iii)] For $p=\infty,$
\begin{eqnarray*}\label{eq15}
\lim_{n\to\infty}\sum_{k}\left| a_{nk}\right|=0.
\end{eqnarray*}
\end{enumerate}
\end{lem}
\begin{lem} \cite{Michael}\label{l7}
$A=(a_{nk})\in (\ell_{1}:\ell_{p})$ if and only if
\begin{eqnarray*}
\sup_{k\in\mathbb{N}_0}\sum_{n}\left| a_{nk} \right|^{p} < \infty
\label{eq16}
\end{eqnarray*}
\end{lem}
\begin{lem}  \cite{Michael}\label{l8}
Let $1 < p < \infty$. Then, $A=(a_{nk})\in \left( \ell_{\infty}: \ell_{p}\right)$ if and only if
\begin{eqnarray*}\label{eq17}
\sup_{K \in \mathcal{F}}\sum_{n}\left| \sum_{k \in K} a_{nk}\right|^{p}< \infty.
\end{eqnarray*}
\end{lem}
\begin{thm}\label{T2}
Let $A=(a_{nk})$ be an infinite matrix. Then, the following statements hold:
\begin{enumerate}
\item[(i)] Let $1 < p < \infty.$ Then, $A \in \left( \ell_{p}^{\lambda}(\widehat{F}):\ell_{\infty}\right) $ if and only if
\begin{eqnarray}\label{eq18}
&&\sum_{j=k+1}^{\infty}a_{nj}f_{j+1}^{2} ~\textrm{ exists for each }~k\in\mathbb{N}_0,\\
\label{eq19}
&&\left\lbrace\frac{\lambda_{k}f^{2}_{k+1}}{(\lambda_{k}-\lambda_{k-1})f_{k}f_{k+1}}{a}_{nk}\right\rbrace\in\ell_{\infty}~\textrm{ for each }~n\in\mathbb{N}_0,\\
\label{eq20}
&&\sup_{k \in \mathbb{N}_0}\sum_{n}\left|e_{nk} \right|^{q}<\infty,\\
\label{eq21}
&&(a_{nk})_{k \in \mathbb{N}_0}\in d_{3}\cap d_{4}\cap d_{5}.
\end{eqnarray}
\item[(ii)] $A \in \left( \ell_{1}^{\lambda}(\widehat{F}): \ell_{\infty}  \right) $ if and only if (\ref{eq18}) and (\ref{eq19}) hold, and
\begin{eqnarray}\label{eq22}
\sup_{n,k \in \mathbb{N}_0} \left| e_{nk}\right| < \infty.
\end{eqnarray}
\item[(iii)] $A\in \left( \ell_{\infty}^{\lambda}(\widehat{F}): \ell_{\infty}  \right) $ if and only if (\ref{eq18}) and (\ref{eq19}) hold, and
\begin{eqnarray}\label{eq23}
&&\sup_{k \in \mathbb{N}_0} \sum_{n}\left| e_{nk}\right| < \infty,\nonumber\\
\label{eq24}
&&\lim\limits_{m \rightarrow \infty} \sum_{n}\left| e_{nk}(m)-e_{nk}\right|=0.
\end{eqnarray}
\end{enumerate}
\end{thm}
\begin{proof}
(i) Suppose that the conditions (\ref{eq18})-(\ref{eq21}) hold and let $x=(x_{k})\in \ell_{p}^{\lambda}(\widehat{F})~\mbox{for}~ 1<p < \infty.$ Then, we have by Theorem \ref{T1} that $(a_{nk})_{k\in\mathbb{N}_0}\in \big[\ell_{p}^{\lambda}(\widehat{F})\big]^{\beta}$ for all $n\in\mathbb{N}_0$ and this implies that $Ax$ exists. Also, it is clear that the associated sequence $y=(y_{k})$ such that $Tx=y$ is in the space $\ell_{p }\subset c_{0}.$

Let us now consider the following equality derived from the $mth$ partial sum of the series $\sum_{k}a_{nk}x_{k}$ by using the relation $y=Ex:$
\begin{eqnarray}\label{c32}
\sum_{k=0}^{m}a_{nk}x_{k}
=\sum_{k=0}^{m-1}e_{nk}(m)y_{k}+\frac{f_{m+1}^{2}\lambda_{m}}{f_{m}f_{m+1}(\lambda_{m}-\lambda_{m-1})}a_{nm}y_{m}
\end{eqnarray}
for all $m,n\in\mathbb{N}_0.$ By using the conditions (\ref{eq18})-(\ref{eq20}), we obtain from (\ref{c32}), as $m\rightarrow\infty,$ that
\begin{eqnarray}\label{eq25}
\sum_{k}a_{nk}x_{k}=\sum_{k}e_{nk}y_{k}
~\mbox{for~ all ~}n \in \mathbb{N}_0.
\end{eqnarray}

Furthermore, $E=(e_{nk}) \in \left( \ell_{p}: \ell_{\infty}\right) $ by Lemma \ref{l3}, we have $Ey\in \ell_{\infty}$. Therefore, one can see by applying H\"{o}lder's inequality that
\begin{eqnarray*}\label{eq26}
\| Ax\|_{\infty}=\sup_{n \in \mathbb{N}_0}\left| \sum_{k}a_{nk}x_{k}\right|
\leq   \sup_{n \in \mathbb{N}_0}\left( \sum_{k}\left| e_{nk}\right|^{q} \right)^{1/q} \left( \sum_{k}\left| y_{k}\right|^{p} \right)^{1/p}< \infty
\end{eqnarray*}
which shows that $Ax \in \ell_{\infty}.$ Hence, $A\in \big( \ell_{p}^{\lambda}(\widehat{F}) : \ell_{\infty}\big).$

Conversely, suppose that $A=(a_{nk})\in \big(\ell_{p}^{\lambda}(\widehat{F}):\ell_{\infty} \big),$ where $1<p<\infty.$ Then,
$(a_{nk})_{k\in\mathbb{N}_0}\in\big[\ell_{p}^{\lambda}(\widehat{F})\big]^{\beta}$ for all $n\in\mathbb{N}_0$ which implies the necessity of the  condition (\ref{eq21}) with Theorem \ref{T1}.
Since $(a_{nk})_{k\in\mathbb{N}_0}\in\big[\ell_{p}^{\lambda}(\widehat{F})\big]^{\beta}$ for all $n\in\mathbb{N}_0,$ (\ref{eq25}) holds for all $x \in\ell_{p}^{\lambda}(\widehat{F})$ and $y\in\ell_{p}.$

Let us now consider the linear functional $f_{n}$ on
$\ell_{p}^{\lambda}(\widehat{F})$ by
\begin{eqnarray*}\label{eq27}
f_{n}(x)=\sum_{k} a_{nk}x_{k} ~\mbox{for~all~} n \in \mathbb{N}_0.
\end{eqnarray*}
Then, since $\ell_{p}^{\lambda}(\widehat{F}) $ and $\ell_{p}$ are norm isomorphic, it should follow with (\ref{eq25})
that
\begin{eqnarray*}\label{eq28}
\| f_{n} \| =\| E_{n} \|_{q}=\left( \sum_{k}\left| e_{nk}\right|^{q} \right)^{1/q} ~\mbox{for~all~} n \in \mathbb{N}_0,
\end{eqnarray*}
where $E_{n}=\left( e_{nk}\right)_{k\in\mathbb{N}_0}\in \ell_{q}. $

This just show that the functional defined by the rows of $A$ on $\ell_{p}^{\lambda}(\widehat{F})$ are pointwise bounded.
Then, we deduce by Banach-Steinhaus Theorem that these functionals are uniformly bounded. Hence, there exists a constant $M>0$ such that $\| f_{n} \| \leq M$ for all $n \in \mathbb{N}_0$ which gives us $\sup_{k \in \mathbb{N}_0}\sum_{n}\left| e_{nk}\right|^{q}< \infty.$

This completes the proof of Part (i).

Similarly, Parts (ii) and (iii) can be proved by Parts (i) and (ii) of Lemma \ref{l3}.
\end{proof}

\begin{thm}\label{T6}
Let $A=(a_{nk})$ be an infinite matrix. Then, the following statements hold:
\begin{enumerate}
\item[(i)] $A \in \left( \ell_{1}^{\lambda}(\widehat{F}): c \right) $ if and only if (\ref{eq18}) and (\ref{eq19}) hold, and
\begin{eqnarray}\label{eq29}
\lim_{n \rightarrow \infty} e_{nk}=\alpha_{k}~\textrm{ for all }~ k \in \mathbb{N}_0.
\end{eqnarray}
\item[(ii)] Let $1 < p < \infty$. Then, $A \in \left( \ell_{p}^{\lambda}(\widehat{F}): c \right)$ if and only if
(\ref{eq18})-(\ref{eq21}) and (\ref{eq29}) hold.
\item[(iii)] $A \in \left( \ell_{\infty}^{\lambda}(\widehat{F}): c \right) $ if and only if (\ref{eq18}), (\ref{eq19}) and (\ref{eq24}) hold, and
\begin{eqnarray*}\label{eq30}
\lim\limits_{n \rightarrow \infty}\sum_{k}\left| e_{nk}-\alpha_{k}\right|=0.
\end{eqnarray*}
\end{enumerate}
\end{thm}
\begin{proof}
Assume that $A$ satisfies the conditions (\ref{eq18})-(\ref{eq21}) and (\ref{eq29}), and $x \in \ell_{p}^{\lambda}(\widehat{F}),$ where $1 <p <\infty.$ Then $Ax$ exists and by using (\ref{eq29}), we have for every $k \in \mathbb{N}_0$ that
$\left| e_{nk} \right|^{q}\rightarrow \left| \alpha_{k} \right|^{q},$ as $n\rightarrow\infty,$ which leads us with (\ref{eq20}) to the following inequality
\begin{eqnarray*}\label{eq31}
\sum_{j=0}^{k}\left| \alpha_{j}\right|^{q} \leq
\sup_{n\in\mathbb{N}_0}\sum_{j=0}^{k}\left| e_{nj}\right|^{q}
=M< \infty
\end{eqnarray*}
which holds for every $k \in \mathbb{N}_0.$ This shows that
$(\alpha_{k})\in \ell_{q}.$ Since $x \in \ell_{p}^{\lambda}(\widehat{F}),$ we have $y \in \ell_{p}.$ Therefore by
H\"{o}lder's inequality,  we derive that $\left( \alpha_{k}y_{k}\right) \in \ell_{1} $ for each $y\in \ell_{p}.$

Now, for any given $\epsilon>0,$ choose a fixed $k_{0}\in \mathbb{N}_0$ such that
\begin{eqnarray*}\label{eq32}
\left( \sum_{k=k_{0}+1}^{\infty}\left| y_{k}\right|^{p} \right)^{1/p} \leq \frac{\epsilon}{4M^{1/q}}.
\end{eqnarray*}
Then, it follows from (\ref{eq29}) that there is $m_{0}\in \mathbb{N}_0$ such that
\begin{eqnarray*}\label{eq33}
\left| \sum_{k=0}^{k_{0}}\left( e_{nk}-\alpha_{k}\right)y_{k} \right|<\frac{\epsilon}{2}~\textrm{ for all }~n \geq n_{0}.
\end{eqnarray*}
Therefore, by using (\ref{eq25}), we get for all $n\geq n_{0}$ that
\begin{eqnarray*}
\left|\sum_{k}a_{nk}x_{k}-\sum_{k}\alpha_{k}y_{k}\right|&=&\left| \sum_{k}\left( e_{nk}-\alpha_{k}\right) y_{k}\right|\\
&\leq&\left| \sum_{k=0}^{k_{0}}\left( e_{nk}-\alpha_{k}\right) y_{k}\right|+
\left| \sum_{k=k_{0}+1}^{\infty}\left( e_{nk}-\alpha_{k}\right) y_{k}\right| \\
&<&\frac{\epsilon}{2}+ \left[ \sum_{k=k_{0}+1}^{\infty}\left( \left|e_{nk} \right| +\left| \alpha_{k}\right| \right)^{q} \right]^{1/q} \left[ \sum_{k=k_{0}+1}^{\infty} \left| y_{k} \right|^{p} \right]^{1/p} \\
&<&\frac{\epsilon}{2}+\frac{\epsilon}{4M^{1/q}}\left[ \left(\sum_{k=k_{0}+1}^{\infty}\left| e_{nk}\right|^{q}  \right)^{1/q}
+\left(\sum_{k=k_{0}+1}^{\infty}\left| \alpha_{k} \right|^{q}  \right)^{1/q}\right] \\
&<&\frac{\epsilon}{2}+\frac{\epsilon}{4M^{1/q}}2M^{1/q}=\epsilon .
\end{eqnarray*}
Hence, $A_{n}(x) \rightarrow \sum_{k} \alpha_{k}y_{k},$ as $n \rightarrow\infty,$ which means that $Ax\in c$ i.e. $A \in \left( \ell_{p}^{\lambda}(\widehat{F}): c \right).$

Conversely let $A \in \left( \ell_{p}^{\lambda}(\widehat{F}): c \right)$ with $1 < p < \infty.$ Since $c \subset \ell_{\infty},$ we have
$A \in \left( \ell_{p}^{\lambda}(\widehat{F}): \ell_{\infty} \right).$ Thus, the necessity of the conditions (\ref{eq18})-(\ref{eq21}) is immediately obtained by Theorem \ref{T2}, which together imply that (\ref{eq25}) holds for all $x \in \ell_{p}^{\lambda}(\widehat{F}).$ Since $Ax \in c$ by the hypothesis, we get by  (\ref{eq25}) that $Ey \in c$ which means that $E=\left( e_{nk}\right) \in \left( \ell_{p}:c\right).$ The necessity of (\ref{eq29}) is immediate by
Lemma \ref{l2}. This completes the proof of Part (i).
\par Since Parts (i) and (iii) can be proved similarly, we omit their proof.
\end{proof}

\begin{thm}\label{T8}
Let $A=(a_{nk})$ be an infinite matrix. Then, the following statements hold:
\begin{enumerate}
\item[(i)] $A \in \left( \ell_{1}^{\lambda}(\widehat{F}): c_{0} \right) $ if and only if (\ref{eq18}) and (\ref{eq19}) hold, and
\begin{eqnarray}\label{eq34}
\lim\limits_{n \rightarrow \infty} e_{nk}=0
~\mbox{for~all~} k \in \mathbb{N}_0.
\end{eqnarray}
\item[(ii)] Let $1 < p < \infty.$ Then, $A \in \left( \ell_{p}^{\lambda}(\widehat{F}): c_{0} \right)$ if and only if
 (\ref{eq18})-(\ref{eq21}) and (\ref{eq34}) hold.
\item[(iii)] $A \in \left( \ell_{\infty}^{\lambda}(\widehat{F}): c_{0} \right) $ if and only if  (\ref{eq18}), (\ref{eq19}) and (\ref{eq24}) hold, and
\begin{eqnarray*}\label{eq35}
\lim\limits_{n \rightarrow \infty}\sum_{k}\left| e_{nk}\right|=0.
\end{eqnarray*}
\end{enumerate}
\end{thm}
\begin{proof}
It is natural that Theorem \ref{T8} can be proved in the same method used in the proof of Theorem \ref{T6} with Lemma \ref{l2} and so, we omit the detail.
\end{proof}

\begin{thm}\label{T9}
Let $A=(a_{nk})$ be an infinite matrix. Then, the following statements hold:
\begin{enumerate}
\item[(i)] $A \in \left( \ell_{1}^{\lambda}(\widehat{F}): \ell_{1} \right) $ if and only if (\ref{eq18}), (\ref{eq19}) and (\ref{eq22})  hold, and
\begin{eqnarray*}\label{eq36}
\sup_{n \in \mathbb{N}_0} \sum_{n}\left| e_{nk}\right|< \infty.
\end{eqnarray*}
\item[(ii)] Let $1 < p < \infty.$ Then, $A \in \left( \ell_{p}^{\lambda}(\widehat{F}): \ell_{1} \right)$ if and only if
 (\ref{eq18})-(\ref{eq21}) hold, and
\begin{eqnarray}\label{eq37}
\sup_{F\in\mathcal{F}}\sum_{k}\left|\sum_{n \in F}e_{nk}\right|^{q}< \infty.
\end{eqnarray}
\item[(iii)] $A \in \left( \ell_{\infty}^{\lambda}(\widehat{F}): \ell_{1} \right) $ if and only if  (\ref{eq18}), (\ref{eq19}) and (\ref{eq24}) hold and
\begin{eqnarray*}\label{eq38}
\sup_{F \in \mathcal{F}} \sum_{k} \left| \sum_{n \in F}e_{nk} \right|< \infty.
\end{eqnarray*}
\end{enumerate}
\end{thm}
\begin{proof}
(ii) Suppose that $A$ satisfies the conditions  (\ref{eq18})-(\ref{eq21})
and (\ref{eq37}) and take any $x \in \ell_{p}^{\lambda}(\widehat{F})$ with $1<p<\infty.$ We have by Theorem \ref{T1} that $(a_{nk})_{k \in \mathbb{N}_0} \in \big[ \ell_{p}^{\lambda}(\widehat{F})\big]^{\beta}$ for all $n\in\mathbb{N}_0$ and this implies that $Ax$ exists. Besides, it follows by combining (\ref{eq37}) with Lemma \ref{l1} that $E\in\left(\ell_{p}:\ell_{1}\right)$ and so, we have $Ey \in \ell_{1}$. Also, we derive from (\ref{eq18})-(\ref{eq21}) that the relation (\ref{eq25}) holds which yields that $Ay\in\ell_{1}$ and so, $A \in \left( \ell_{p}^{\lambda}(\widehat{F}): \ell_{1}\right).$

Conversely, assume that $A \in \left( \ell_{p}^{\lambda}(\widehat{F}): \ell_{1}\right)$ with $1 <p < \infty.$ Since $\ell_{1}\subset \ell_{\infty},$ we get $A \in \left( \ell_{p}^{\lambda}(\widehat{F}): \ell_{\infty}\right).$ Thus, Theorem \ref{T2} implies the necessity of the conditions (\ref{eq18})-(\ref{eq21}) which leads to the relation (\ref{eq25}). Since $Ax \in \ell_{1},$ we deduce by (\ref{eq25}) that $Ey \in \ell_{1}$ which means $E\in \left( \ell_{p}: \ell_{1} \right).$ Now, the necessity of (\ref{eq37}) is immediate by Lemma \ref{l1}. This completes the proof of Part (ii).

 Parts (i) and (iii) can be proved in the similar way, so we omit the details.
\end{proof}

\begin{thm}\label{T10}
$A=(a_{nk})\in \left( \ell_{1}^{\lambda}(\widehat{F}): \ell_{p}\right)$ if and only if  (\ref{eq18}), (\ref{eq19}) hold, and
\begin{eqnarray}\label{eq39}
\sup_{k \in \mathbb{N}_0} \sum_{n} \left| e_{nk}\right|^{p} < \infty.
\end{eqnarray}
\end{thm}
\begin{proof}
Suppose that the conditions (\ref{eq18}), (\ref{eq19}) and (\ref{eq39}) hold, and take $x\in\ell_{1}^{\lambda}(\widehat{F}).$ Then, we have by Theorem \ref{T1} that $(a_{nk})_{k \in \mathbb{N}_0} \in \big[\ell_{1}^{\lambda}(\widehat{F})\big]^{\beta}$ for all $n\in\mathbb{N}_0$ which implies that $Ax$ exists. From (\ref{eq39}), we have
\begin{eqnarray*}\label{eq40}
\sup_{k \in \mathbb{N}_0}\left|e_{nk}\right|\leq\sup_{k\in\mathbb{N}_0}\left(\sum_{n}\left|e_{nk}\right|^{p} \right)^{1/p} < \infty~\textrm{ for each }~ n\in \mathbb{N}_0.
\end{eqnarray*}
Hence, $\sum_{n}\left| e_{nk}\right|$ absolutely converges for each fixed $n \in \mathbb{N}_0.$ Since (\ref{eq18}) and (\ref{eq19}) hold, therefore as $m \rightarrow \infty$ in (\ref{c32}), the relation (\ref{eq25}) holds. Thus, by applying Minkowski's inequality and using (\ref{eq25}) and (\ref{eq39}) we
obtain that
\begin{eqnarray*}\label{eq41}
\left( \sum_{n}\left| \sum_{k} a_{nk}x_{k}\right|^{p} \right)^{1/p}
=\left( \sum_{n}\left| \sum_{k} e_{nk}y_{k}\right|^{p} \right)^{1/p}
\leq \sum_{k} \left| y_{k}\right|
\left( \sum_{n}\left|  e_{nk}\right|^{p} \right)^{1/p}
< \infty
\end{eqnarray*}
which means $Ax\in\ell_{p},$ that is, $A=(a_{nk})\in \left( \ell_{1}^{\lambda}(\widehat{F}): \ell_{p}\right).$

Conversely, let $A \in \left( \ell_{1}^{\lambda}(\widehat{F}): \ell_{p}\right).$ Since
$\ell_{p} \subset \ell_{\infty}$, $A \in \left( \ell_{1}^{\lambda}(\widehat{F}): \ell_{\infty}\right).$
Thus, Theorem \ref{T2} gives the necessity of (\ref{eq18}) and (\ref{eq19}) by the relation (\ref{eq25}). Since $Ax \in \ell_{p},$ we deduce by (\ref{eq25}) that $Ey \in \ell_{p}$ which means $E\in \left( \ell_{1}: \ell_{p}\right).$ Now, the  necessity of (\ref{eq39}) is immediate by Lemma \ref{l7}. This step completes the proof.
\end{proof}

\begin{thm}\label{T7}
Let $1 < p < \infty.$ Then, $A=(a_{nk})\in \left( \ell_{\infty}^{\lambda}(\widehat{F}): \ell_{p}\right)$
if and only if (\ref{eq18}) and (\ref{eq19})  hold, and
\begin{eqnarray*}
&&\sum_{k} \left| e_{nk}\right| ~\mbox{converges~for~all}~ n \in \mathbb{N}_0,\\
&&\sup_{K\in \mathcal{F}}\sum_{n} \left| \sum_{k \in K} e_{nk}\right|^{p} < \infty.
\end{eqnarray*}
\end{thm}
\begin{proof}
This is obtained in the same way as done in Theorem \ref{T10} by Lemma \ref{l8}. So, we omit the details.
\end{proof}
\begin{lem} \cite{Altay,16} \label{l10}
Let $X$ and $Y$ be any two sequence spaces, $A$ be an infinite matrix and $B$ be a triangle. Then, $A \in \left(X: Y_{B}\right)$
if and only if $BA \in \left(X :Y\right).$
\end{lem}
\begin{cor}
Let $A=(a_{nk})$ be an infinite matrix and define the matrix $C=(c_{nk})$ by
\[c_{nk}=\frac{1}{\lambda_{n}}\sum_{i=0}^{n}\left(\lambda_{i}-\lambda_{i-1} \right)\left( \frac{f_{i}}{f_{i+1}}a_{ik}-\frac{f_{i+1}}{f_{i}}a_{i-1,k}\right)\]
for all $n,k \in \mathbb{N}_0.$ By applying Lemma \ref{l10} we get, $A$ belongs to any one of the classes
 $\left( c_{0}:\ell_{p}^{\lambda}(\widehat{F}) \right),$
 $\left( c:\ell_{p}^{\lambda}(\widehat{F}) \right),$
 $\left( \ell_{\infty}:\ell_{p}^{\lambda}(\widehat{F}) \right),$
 $\left( \ell_{1}:\ell_{p}^{\lambda}(\widehat{F}) \right),$
 $\left( \ell_{p}:\ell_{1}^{\lambda}(\widehat{F}) \right)$ and
 $\left( \ell_{p}:\ell_{\infty}^{\lambda}(\widehat{F}) \right)$
 if and only if the matrix $C$ belongs to the classes
 $\left( c_{0}:\ell_{p} \right),$
  $\left( c:\ell_{p} \right),$
  $\left( \ell_{\infty}:\ell_{p} \right),$
  $\left( \ell_{1}:\ell_{p} \right),$
  $\left( \ell_{p}: \ell_{1}\right)$ and
  $\left( \ell_{p}:\ell_{\infty} \right),$ respectively; where $1\leq p\leq\infty.$
\end{cor}

 \begin{cor}
 Let $A=(a_{nk})$ be an infinite matrix and define the matrix $C'=(c'_{nk})$ by
 $$ c'_{nk}=\frac{1}{\lambda'_{n}}\sum_{i=0}^{n}\left(\lambda'_{i}-\lambda'_{i-1} \right)\left( \frac{f_{i}}{f_{i+1}}a_{ik}-\frac{f_{i+1}}{f_{i}}a_{i-1,k}\right); \; (n,k \in \mathbb{N}_0).  $$
Then, the necessary and sufficient conditions such that
  $A$ belongs to any one of the classes
  $\left(\ell_{p}^{\lambda}(\widehat{F}): \ell_{\infty}^{\lambda'}(\widehat{F}) \right),$
  $\left(\ell_{p}^{\lambda}(\widehat{F}): c_{0}^{\lambda'}(\widehat{F})  \right),$
  $\left(\ell_{p}^{\lambda}(\widehat{F}): c^{\lambda'}(\widehat{F})  \right),$
   $\left(\ell_{p}^{\lambda}(\widehat{F}): \ell_{1}^{\lambda'}(\widehat{F})\right),$
   $\left(\ell_{1}^{\lambda}(\widehat{F}): \ell_{p}^{\lambda'}(\widehat{F})\right)$
   and
   $\left(\ell_{\infty}^{\lambda}(\widehat{F}): \ell_{p}^{\lambda'}(\widehat{F}) \right),$
   where $1\leq p \leq \infty$ are obtained from the respective Theorems \ref{T2} to \ref{T7} by replacing the entries of matrix $A$ by those of $C$ and $\lambda'=\left(\lambda'_{k}\right) $ is a strictly increasing sequence of positive reals tending to infinity and $E=\left(e_{nk}\right)$ is a triangle defined by (\ref{f}) with $\lambda'$ instead of $\lambda.$
 \end{cor}

 \section{Some geometric properties of $\ell_{p}^{\lambda}(\widehat{F})$, $(1 < p < \infty)$}

 In this section, we study some geometric properties of the space
 $\ell_{p}^{\lambda}(\widehat{F}),$ where $1<p<\infty.$

A Banach space $X$ is said to have the \textit{Banach-Saks property} if every bounded sequence $(x_{n})$ in $X$ admits a subsequence $(z_{n})$ such that the sequence $C_{1}(z)$ is convergent in the norm of $X,$ (see \cite{J}), where $C_{1}(z)=(t_{n})$ defined by
\begin{eqnarray*}
t_{n}=\frac{1}{n+1}\left( z_{0}+z_{1}+\cdots +z_{n}\right)
\end{eqnarray*}
for all $n\in\mathbb{N}_0.$

A Banach space $X$ is said to have the \textit{weak Banach-Saks property} whenever, given any weakly null sequence $(x_{n})$ in $X,$ there exists a subsequence $(z_{n})$ of $(x_{n})$ such that the sequence $(t_{n})$ is strongly convergent to zero.
\par  Garc\'{\i}a-Falset \cite{G} introduced the following coefficient:
\begin{eqnarray*}
R(X)=\sup\left\lbrace \liminf\limits_{n \rightarrow \infty} \| x_{n}-x\| : (x_{n}) \subset B(X),~ x_{n} \rightarrow 0 \mbox{~weakly}, ~x \in  B(X) \right\rbrace,
\end{eqnarray*}
where $B(X)$ denotes the unit ball of $X.$
\begin{rem}\cite{G1}
A Banach space $X$ with $R(X)<2$ has the weak fixed point property.
\end{rem}

Let $1< p< \infty.$ A Banach space is said to have the Banach-Saks type $p$ if every weakly null sequence $(x_{k})$ has a subsequence $(x_{k_{l}})$ such that for some $C>0,$
\begin{eqnarray*}
\left\|\sum_{l=0}^{n}x_{k_{l}}\right\|<C(n+1)^{1/n}
\end{eqnarray*}
for all $n\in\mathbb{N}_0$ (see \cite{K1}).

\begin{thm}
Let $1<p< \infty.$ The space $\ell_{p}^{\lambda}(\widehat{F})$ has Banach-Saks type $p.$
\end{thm}
\begin{proof}
Let $(\epsilon_{n})$ be a sequence of positive numbers for which
$\sum_{n=1}^{\infty}\epsilon_{n} \leq 1/2.$ Let $(x_{n})$ be a weakly null sequence in $B\left( \ell_{p}^{\lambda}(\widehat{F})\right).$ Let $u_{0}=x_{0}$ and $u_{1}=x_{n_{1}}.$
Then, there exists $t_{1}\in \mathbb{N}_0$ such that
\begin{eqnarray*}
\left\|\sum_{i=t_{1}+1}^{\infty} u_{1}(i)e^{(i)}\right\|_{\ell_{p}^{\lambda}(\widehat{F})}< \epsilon_{1}.
\end{eqnarray*}
Since $(x_{n})$ is a weakly null sequence implies $x_{n}\rightarrow 0$ (coordinatewise), there exists $n_{2}\in \mathbb{N}_0$ such that
\begin{eqnarray*}
\left\|\sum_{i=0}^{t_{1}} x_{n}(i)e^{(i)}\right\|_{\ell_{p}^{\lambda}(\widehat{F})}< \epsilon_{1},
\end{eqnarray*}
where $n \geq n_{2}.$ Set $u_{2}=x_{n_{2}}.$ Then, there exists $t_{2}> t_{1}$ such that
\begin{eqnarray*}
\left\|\sum_{i=t_{2}+1}^{\infty} u_{2}(i)e^{(i)}\right\|_{\ell_{p}^{\lambda}(\widehat{F})}< \epsilon_{2}.
\end{eqnarray*}
By using the fact that $x_{n}\rightarrow 0$ with respect to coordinatewise, there exists $n_{3}> n_{2}$ such that
\begin{eqnarray*}
\left\|\sum_{i=0}^{t_{2}} x_{n}(i)e^{(i)}\right\|_{\ell_{p}^{\lambda}(\widehat{F})}< \epsilon_{2},
\end{eqnarray*}
where $n \geq n_{3}.$ If we continue this process, we can find two increasing sequences $(t_{i})$ and $(n_{i})$ of natural numbers such that
\begin{eqnarray*}
\left\|\sum_{i=0}^{t_{j}} x_{n}(i)e^{(i)}\right\|_{\ell_{p}^{\lambda}(\widehat{F})}< \epsilon_{j}
\end{eqnarray*}
for each $n \geq n_{j+1}$ and
\begin{eqnarray*}
\left\|\sum_{i=t_{j}+1}^{\infty} u_{j}(i)e^{(i)}\right\|_{\ell_{p}^{\lambda}(\widehat{F})}< \epsilon_{j},
\end{eqnarray*}
where $u_{j}=x_{n_{j}}.$ Hence,
\begin{eqnarray*}
\left\|\sum_{j=0}^{n} u_{j}\right\|_{\ell_{p}^{\lambda}(\widehat{F})}&=&\left\|\sum_{j=0}^{n} \left[\sum_{i=0}^{t_{j-1}} u_{j}(i)e^{(i)} +
\sum_{i=t_{j-1}+1}^{t_{j}} u_{j}(i)e^{(i)} +
\sum_{i=t_{j}+1}^{\infty} u_{j}(i)e^{(i)} \right]\right\|_{\ell_{p}^{\lambda}(\widehat{F})} \\
&\leq&\left\|\sum_{j=0}^{n} \left[\sum_{i=t_{j-1}+1}^{t_{j}} u_{j}(i)e^{(i)}\right]\right\|_{\ell_{p}^{\lambda}(\widehat{F})}
 +2 \sum_{j=0}^{n}\epsilon_{j}.
\end{eqnarray*}

On the other hand, we have $\| x \|_{\ell_{p}^{\lambda}(\widehat{F})}<1.$ Thus,
$\| x \|^{p}_{\ell_{p}^{\lambda}(\widehat{F})}<1$ and we have,
\[\left\|\sum_{j=0}^{n} \left(\sum_{i=t_{j-1}+1}^{t_{j}} u_{j}(i)e^{(i)}\right)\right\|_{\ell_{p}^{\lambda}(\widehat{F})}
\leq \sum_{j=0}^{n}\sum_{i=t_{j-1}+1}^{t_{j}}
\left| E_{i}\left( u_{i}\right)\right|^{p}
\leq \sum_{j=0}^{n}\sum_{i=0}^{\infty}
\left| E_{i}\left( u_{i}\right)\right|^{p} \leq (n+1).
\]
Therefore
\[\left\|\sum_{j=0}^{n} \left(\sum_{i=t_{j-1}+1}^{t_{j}} u_{j}(i)e^{(i)}\right)\right\|_{\ell_{p}^{\lambda}(\widehat{F})}\leq
(n+1)^{1/p} .\]
By using the fact that $1 \leq (n+1)^{1/p}$ for all $n \in \mathbb{N}_0$ and $1< p< \infty,$ we have
\[\left\|\sum_{j=0}^{n} u_{j}\right\|_{\ell_{p}^{\lambda}(\widehat{F})}
 \leq (n+1)^{1/p}+1 \leq 2(n+1)^{1/p}.\]
 Therefore the space $\ell_{p}^{\lambda}(\widehat{F})$ has \textit{Banach-Saks} type $p.$
\end{proof}
\begin{rem}
Note that $R\left( \ell_{p}^{\lambda}(\widehat{F})\right)
=R\left( \ell_{p}\right) =2^{1/p},$ since $\ell_{p}^{\lambda}(\widehat{F})$  is linearly isomorphic to $\ell_{p}.$
\end{rem}
\begin{thm}
The space $\ell_{p}^{\lambda}(\widehat{F})$ has the weak fixed point property, where $1<p<\infty.$
\end{thm}

\section{ Compact operators on the spaces $\ell_{p}^{\lambda}(\widehat{F}),$  $(1\leq p \leq \infty) $ }

In this section, we establish some estimates for the operator norms and the Hausdorff measures of non-compactness of certain matrix operators on the spaces
$\ell_{p}^{\lambda}(\widehat{F})$ and $\ell_{\infty}^{\lambda}(\widehat{F}).$ Further, by using the Hausdorff measure of non-compactness, we characterize some classes of compact operators on these spaces.

 For our investigations, we need the following results:
\begin{thm} \cite{cp3, cp6} \label{o1}
Let $X$ and $Y$ be FK spaces. Then $(X:Y) \subset B(X:Y),$
that is, every $A \in (X:Y)$ defines a linear operator $L_{A}\in B(X:Y),$ where $L_{A}x=Ax$ and $x\in X.$
\end{thm}

\begin{thm} \cite{cp4}\label{o2}
Let $X \supset \phi$ and $Y$ be $BK$-spaces. Then $A \in(X:\ell_{\infty})$ if and only if
 \[\|A\|_{X}^{*}=\sup_{n\in\mathbb{N}_0}\|A_{n}\|_{X}^{*}< \infty.\]
Furthermore, if  $A \in(X:\ell_{\infty})$ then it follows that $\| L_{A}\|=\| A\|_{X}^{*}.$
\end{thm}

\begin{thm} \cite{cp2} \label{o5}
Let $X$ be a $BK$-space. Then $A \in(X:\ell_{1})$ if and only if
\[\| A \|^{*}_{(X:\ell_{1})}
=\sup_{N \subset \mathbb{N}_0}\left\|\left( \sum_{n \in N} a_{nk}\right)_{k\in\mathbb{N}_0}\right\|_{X}^{*} < \infty,\] where $N$ is finite. Moreover, if $A \in \left(X:\ell_{1}\right) $ then
$\| A \|^{*}_{(X:\ell_{1})}\leq\| L_{A}\|\leq 4\| A \|^{*}_{(X:\ell_{1})}.$
\end{thm}
Throughout, let $T=(t_{nk})_{k,n\in\mathbb{N}_0}$ be a triangle, that is, $t_{nk}=0$ for $k>n$ and $t_{nn}\neq 0$ for all $n\in\mathbb{N}_0,$ $S$ its inverse and $R=S^{t},$ the transpose of $S.$ The following results are known:
\begin{thm} \cite{cp1,cp6} \label{o3}
Let $\left( X, \|\cdot\|\right) $ be a $BK$-space. Then $X_{T}$ is a $BK$-space with
$\|\cdot \|_{T}= \|T\cdot\|.$
\end{thm}
\begin{rem} \cite{cp1}
The matrix domain $X_{T}$ of a normed sequence space $X$ has a basis if and only if $X$ has a basis.
\end{rem}

\begin{thm} \cite{cp1}\label{T3}
Let $X$ be a $BK$-space with AK and $R=S^{t}.$ If $a=(a_{k})\in \left( X_{T}\right)^{\beta},$ then
$\sum\limits_{k}a_{k}x_{k}=\sum\limits_{k}R_{k}(a)T_{k}(x)$ for all $x=(x_{k})\in X_{T}.$
\end{thm}

\begin{rem} \cite{cp1}
The conclusion of Theorem \ref{T3} holds for $X=c$ and $X=\ell_{\infty}.$
\end{rem}

\begin{thm} \cite{cp3}
Let $X$ and $Y$ be Banach spaces, $S_{X}=\left\lbrace x \in X : \| x \|=1 \right\rbrace ,$
$K_{X}=\left\lbrace x \in X : \| x \| \leq 1 \right\rbrace$ and $A\in B(X:Y).$ Then, the Hausdorff measure of non-compactness $\|A\|_{\chi}$ of a compact operator $A$ is given by $\| A \|_{\chi}=\chi\left( AK \right)=\chi\left( AS \right).$
\end{thm}
Furthermore, $A$ is compact if and only if $\| A \|_{\chi}=0,$ (see \cite{cp3}). The Hausdorff measure of non-compactness satisfies the inequality
 $\| A \|_{\chi} \leq  \| A \|,$ (see \cite{cp3}).
\begin{thm} \cite{cp3} \label{o7}
Let $X$ be a Banach space with Schauder  basis $\left\lbrace
e_{1}, e_{2},\ldots \right\rbrace,$ $Q$ be a bounded subset of $X$
and $P_{n}: X\rightarrow X$ be the projector onto the linear span of $\left\lbrace
e_{1}, e_{2},\ldots,e_{n} \right\rbrace.$ Then,
\begin{eqnarray*}
\frac{1}{a} \limsup\limits_{n \rightarrow \infty}\left[\sup_{x \in Q}\| (I-P_{n})x\|\right]\leq \chi(Q)
\leq \limsup\limits_{n \rightarrow \infty}\left[\sup_{x \in Q}\|(I-P_{n})x\|\right],
\end{eqnarray*}
where $a=\limsup\limits_{n\rightarrow\infty}\|I-P_{n}\|.$
\end{thm}

\begin{thm} \cite{cp5} \label{o6}
Let $X$ be any of the spaces $\ell_{p}$ or $c_{0}$ and $Q$ be a bounded subset of a normed space $X.$ If  $P_{n}: X\rightarrow X$ is an operator defined by $P_{n}(x)=\left( x_{0},x_{1},\ldots,x_{n}, 0, 0,\ldots \right), $ then
\begin{eqnarray*}
\chi(Q)= \lim_{n\to\infty} \left(\sup_{x\in Q}\|(I-P_{n})x\|\right).
\end{eqnarray*}
\end{thm}

\begin{thm} \cite{cp7}
Let $X$ be a normed sequence space and $\chi_{T}$ and $\chi$ denote the Hausdorff measures of non-compactness on $M_{X_{T}}$ and $M_{X},$ the collection of all bounded sets in $X_{T}$ and $X,$ respectively. Then $\chi_{T}(Q)=\chi(T(Q))$ for all $Q \in M_{X_{T}}.$
\end{thm}

\begin{lem} \cite{cp3} \label{l9}
 Let $X$ denote any of the spaces $c_{0},$ $c$ or $\ell_{\infty}.$ Then, $X^{\beta}=\ell_{1}$ and $\| a \|_{X}^{*}
=\| a\|_{1}$ for all  $a \in \ell_{1}.$
\end{lem}

If $A=(a_{nk})_{k,n\in\mathbb{N}}$ is an infinite matrix, and $N$ is any finite subset of $\mathbb{N}_0,$ we write
$b^{(N)}=\left(b^{(N)}_{k} \right)_{k\in\mathbb{N}_0}=(\sum\limits_{n\in N}a_{nk})_{k\in\mathbb{N}_0}.$ Also, we have $\widehat{a}_{nk}=R_{k}A_{n}$ for all $n,k\in\mathbb{N}_0.$

\begin{thm}\cite{cp}
Let $X$ be any of the spaces $\ell_{p}$ with $1\leq p\leq\infty$ or $c_{0}.$ Then, the following statements hold:
\begin{enumerate}
\item[(a)] Let $Y\in\{c_{0}, c, \ell_{\infty}\}.$ If $A \in( X_{T}:Y),$ then we put
\begin{eqnarray*}
\|A\|_{(X_{T}:\ell_{\infty})}=\sup_{n\in\mathbb{N}_0}\|E_{n}\|_{q}=\left\{\begin{array}{ccl}
\underset{n\in\mathbb{N}_0}{\sup}\sum\limits_{k}\left|\widehat{a}_{nk}\right|&,& (X=c_{0}, \ell_{\infty}),   \\
\underset{n\in\mathbb{N}_0}{\sup}\left(\sum\limits_{k}\left|\widehat{a}_{nk}\right|^{q}\right)^{1/q}&,&(X=\ell_{p} ~ for~ 1<p<\infty),  \\
\underset{k,n\in\mathbb{N}_0}{\sup}\left|\widehat{a}_{nk}\right|&,&( X=\ell_{1}).
\end{array}\right.
\end{eqnarray*}
Therefore, we have $\|L_{A}\|=\|A\|_{(X_{T}:\ell_{\infty})}.$
\item[(b)] Let $Y=\ell_{1}.$ If $A \in (X_{T}:\ell_{1}).$ Then we put
\begin{eqnarray*}
\|A\|_{(X_{T}:\ell_{1})}=\sup_{N}\|\widehat{b}^{(N)}\|_{q}=\left\{\begin{array}{ccl}
\underset{N\in\mathcal{F}}{\sup}\sum\limits_{k}\left|\sum_{n\in N}\widehat{a}_{nk}\right|&,&(X=c_{0}, \ell_{\infty}),   \\
\underset{N\in\mathcal{F}}{\sup}\left(\sum\limits_{k}\left|\sum\limits_{n\in N}\widehat{a}_{nk}\right|^{q}\right)^{1/q}&,&(X=\ell_{p} ~\textrm{ for }~ 1<p<\infty)
\end{array}\right.
\end{eqnarray*}
and
\[\|A\|_{((\ell_{1})_{T}:\ell_{1})}=\sup_{k\in\mathbb{N}_0}\|E^{k}\|_{1}
=\sup_{k\in\mathbb{N}_0}\sum_{n=0}^{\infty} \left|\widehat{a}_{nk} \right|.\]
If $X=\ell_{1},$ then $\|L_{A}\|=\|A\|_{((\ell_{1})_{T}:\ell_{1})}$ holds, otherwise we have $\|A\|_{(X_{T}:\ell_{1})}\leq
\|L_{A}\|\leq 4\|A\|_{(X_{T}:\ell_{1})}.$
\end{enumerate}
\label{com1}
\end{thm}

By $\mathbb{N}_{r},$ we denote the subset of $\mathbb{N}_0$ with the elements that are greater than or equal to $r\in\mathbb{N}_0$ and $\sup_{\mathbb{N}_{r}}$ for the supremum taken over finite subset of $\mathbb{N}_{r}.$
\begin{thm}\cite{cp}
Let $A=(a_{nk})$ be an infinite matrix and $1 \leq p \leq \infty.$ Then, the following statements hold:
\begin{enumerate}
\item[(a)] If $A \in \left( (\ell_{p})_{T}:c_{0}\right)$ or $A \in \left( (c_{0})_{T}:c_{0}\right),$ then we have
\begin{eqnarray*}
\|L_{A}\|_{\chi}=\lim_{r\rightarrow\infty}\left(\sup_{n\in\mathbb{N}_{r}} \| E_{n}\|_{q}\right)= \left\{\begin{array}{ccl}
\underset{r\to\infty}{\lim}\left(\underset{n\in\mathbb{N}_{r}}{\sup}\sum\limits_{k}\left|\widehat{a}_{nk}\right|\right)&,&(X=c_{0},\ell_{\infty}),   \\
\underset{r\to\infty}{\lim}\left[\underset{n\in\mathbb{N}_{r}}{\sup}\left(\sum\limits_{k}\left|\widehat{a}_{nk}\right|^{q}\right)^{1/q}\right]&,&(X=\ell_{p} ~ for~ 1<p<\infty), \\
\underset{r\to\infty}{\lim}\left(\underset{n\in\mathbb{N}_{r},k\in\mathbb{N}}{\sup}\left|\widehat{a}_{nk}\right|\right)&,&(X=\ell_{1}).
\end{array}\right.
\end{eqnarray*}
\item[(b)] If $A \in \left( (\ell_{p})_{T}:\ell_{1}\right)~( 1<p \leq \infty)$ or $A \in((c_{0})_{T}:\ell_{1}),$ then we have
\[\lim_{r\rightarrow\infty}\left(\sup_{\mathbb{N}_{r}}\left\|\sum_{n\in\mathbb{N}_{r}}E_{n}\right\|_{q}\right)
 \leq \|L_{A}\|_{\chi}\leq4\lim_{r\rightarrow\infty}\left(\sup_{\mathbb{N}_{r}}\left\|\sum_{n\in\mathbb{N}_{r}}E_{n}\right\|_{q}\right),
\]
if $A\in\left((\ell_{1})_{T}:\ell_{1}\right),$ then we have
\[\|L_{A}\|_{\chi}=\lim_{r \rightarrow\infty}\left[\sup_{k\in\mathbb{N}_0}\left(\sum_{n=r}^{\infty}\left|\widehat{a}_{nk}\right|\right)\right].
\]
\item[(c)] If $A \in((\ell_{p})_{T}:c)$ or $A \in((c_{0})_{T}:c),$ then we have
\[
 \frac{1}{2}\lim_{r\rightarrow \infty}\left(\sup_{n\in\mathbb{N}_{r}}\|E_{n}- \widehat{\alpha}\|_{q}\right)
 \leq \| L_{A} \|_{\chi} \leq
\lim_{r\rightarrow\infty}\left(\sup_{n\in\mathbb{N}_{r}}\|E_{n}- \widehat{\alpha}\|_{q}\right),
\]
where $\widehat{\alpha}=\left(\widehat{\alpha}_{k}\right)_{k\in\mathbb{N}_0}$
with $\widehat{a}_{nk}\to\widehat{\alpha}_{k},$ as $n\to\infty,$ for every $k\in\mathbb{N}_0.$
\end{enumerate}
\label{com2}
\end{thm}
\begin{thm}
Let $A=(a_{nk})$ be an infinite matrix and define the matrix $E=(e_{nk})_{k,n\in\mathbb{N}}$ by (\ref{eq400}). Then, the following statements hold:
\begin{enumerate}
\item[(a)] Let $Y\in\{c_{0}, c, \ell_{\infty}\}.$ If $A \in \left( \ell_{p}^{\lambda}(\widehat{F}):Y \right)$ with $1\leq p\leq\infty,$ then we have
\begin{eqnarray*}
\|A\|_{(\ell_{p}^{\lambda}(\widehat{F}):\ell_{\infty})}=\left\{\begin{array}{ccl}
\underset{n\in\mathbb{N}_0}{\sup}\sum\limits_{k}\left|\widehat{a}_{nk}\right|&,& (p=\infty),   \\
\underset{n\in\mathbb{N}_0}{\sup}\left(\sum\limits_{k}\left|\widehat{a}_{nk}\right|^{q}\right)^{1/q}&,&(1<p<\infty),  \\
\underset{k,n\in\mathbb{N}_0}{\sup}\left|\widehat{a}_{nk}\right|&,&(p=1).
\end{array}\right.
\end{eqnarray*}
Then, we have $\|L_{A}\|=\|A\|_{(\ell_{p}^{\lambda}(\widehat{F}):\ell_{\infty})}.$
\item[(b)] If $A\in\left(\ell_{p}^{\lambda}(\widehat{F}):\ell_{1}\right)$ with $1\leq p\leq\infty,$ then we have
\begin{eqnarray*}
\|A\|_{(\ell_{p}^{\lambda}(\widehat{F}):\ell_{1})}=\left\{\begin{array}{ccl}
\underset{N\in\mathcal{F}}{\sup}\sum\limits_{k}\left|\sum\limits_{n \in N}\widehat{a}_{nk}\right|&,&(p=\infty),   \\
\underset{N\in\mathcal{F}}{\sup}\left(\sum\limits_{k}\left|\sum\limits_{n\in N}\widehat{a}_{nk}\right|^{q}\right)^{1/q}&,&(1<p<\infty),  \\
\underset{k\in\mathbb{N}_0}{\sup}\sum\limits_{n=0}^{\infty}\left|\widehat{a}_{nk}\right|&,&(p=1).
\end{array}\right.
\end{eqnarray*}
Then, for $p=1,$ $\|L_{A}\|=\|A\|_{(\ell_{p}^{\lambda}(\widehat{F}):\ell_{1})}$ holds, otherwise $\|A\| _{(\ell_{p}^{\lambda}(\widehat{F}):\ell_{1})} \leq \| L_{A} \|\leq 4\| A \| _{(\ell_{p}^{\lambda}(\widehat{F}):\ell_{1})}.$
\end{enumerate}
\label{TC0}
\end{thm}
\begin{proof}
Suppose that $A\in \left( \ell_{p}^{\lambda}(\widehat{F}):Y\right).$ Then, we have $A_{n}\in\big[\ell_{p}^{\lambda}(\widehat{F})\big]^{\beta}$ for all $n\in\mathbb{N}_0$ and it follows from Theorem \ref{T3} that
\begin{eqnarray*}\label{c39}
A_{n}(x)=\sum_{k}a_{nk}x_{k}
=\sum_{k}R_{k}(A_{n})T_{k}(x)
\end{eqnarray*}
for all $x \in c^{\lambda}(\widehat{F})$ and $n\in\mathbb{N}_0,$ where $R_{k}(A_{n})=\sum\limits_{j}r_{kj}a_{nj}=\sum\limits_{j}s_{jk}a_{nj}.$ Here $T=E$ and $S=E^{-1}.$ Therefore we have $R_{k}(A_{n})=e_{nk}$ for all $k,n\in\mathbb{N}_0.$

Proof of Part (a) can be obtained by applying Theorem \ref{com1}. Since the proof of Part (b) is similar to the proof of Part (a), we omit the details.
\end{proof}

\begin{thm}
 Let $A=(a_{nk})$ be an infinite matrix and $E=(e_{nk})_{k,n\in\mathbb{N}_0}$ be defined by (\ref{eq400}). Then, for $1\leq p\leq\infty,$ we have
 \begin{enumerate}
\item[(a)] If $A \in \left( \ell_{p}^{\lambda}(\widehat{F}):c_{0} \right),$ then we have
\begin{eqnarray*}
\|L_{A}\|_{\chi}=\left\{\begin{array}{ccl}
\lim\limits_{r \rightarrow\infty}\left(\sup\limits_{n\in\mathbb{N}_{r}}\sum\limits_{k}\left|\widehat{a}_{nk}\right|\right)&,&(p=\infty),   \\
\lim\limits_{r\rightarrow\infty}\left(\sup\limits_{n\in\mathbb{N}_{r}}\left(\sum\limits_{k}\left|\widehat{a}_{nk}\right|^{q}\right)^{1/q}\right)&,&(1<p<\infty),  \\
\lim\limits_{r\rightarrow\infty}\left(\sup\limits_{n\geq r,k \geq 0}\left|\widehat{a}_{nk}\right|\right)&,&(p=1).
\end{array}\right.
\end{eqnarray*}
\item[(b)] If $A\in\left(\ell_{p}^{\lambda}(\widehat{F}):\ell_{1}\right)$ with $1<p\leq\infty$, then we have
\[\lim\limits_{r \rightarrow \infty}\left(
\sup_{\mathbb{N}_{r}}\left\|\sum_{n \in \mathbb{N}_{r}}E_{n}\right\|_{q}\right)\leq\|L_{A}\|_{\chi}\leq4\lim_{r\rightarrow\infty}\left(
\sup_{\mathbb{N}_{r}}\left\|\sum_{n \in \mathbb{N}_{r}}E_{n}\right\|_{q}\right).
\]
If $A \in \left(\ell_{1}^{\lambda}(\widehat{F}):\ell_{1}\right),$ then we have
\[\|L_{A}\|_{\chi}=\lim_{r\rightarrow\infty}\left[\sup_{k\in\mathbb{N}_0}\left(\sum_{n=r}^{\infty}\left|\widehat{a}_{nk}\right|\right)\right].
         \]
\item[(c)] If $A \in \left(\ell_{p}^{\lambda}(\widehat{F}):c \right),$ then we have
       \[\frac{1}{2}\lim\limits_{r \rightarrow \infty}\left(
          \sup_{n\in\mathbb{N}_{r}} \|E_{n}-\widehat{\alpha} \|_{q}\right)
           \leq \| L_{A} \| _{\chi} \leq
           \lim\limits_{r \rightarrow \infty}\left(
                    \sup_{n\in\mathbb{N}_{r}} \|E_{n}-\widehat{\alpha} \|_{q}\right),
          \]
 where $\widehat{\alpha}=\left( \widehat{\alpha}_{k}\right)_{k\in\mathbb{N}} $ with $\widehat{\alpha}_{k}=\lim\limits_{n \rightarrow \infty} \widehat{a}_{nk}$ for all $k\in\mathbb{N}_0.$
 \end{enumerate}
\label{TC1}
\end{thm}
\begin{proof}
Proof of Theorem \ref{TC1} can be given in the same way as that of Theorem \ref{TC0} by applying Theorem \ref{com2} instead of
Theorem \ref{com1}.
\end{proof}
 \begin{cor}
 Let $1 \leq p \leq \infty.$ Then the following statements hold:
 \begin{enumerate}
 \item[(a)] If $A \in \left( \ell_{p}^{\lambda}(\widehat{F}):c_{0}\right),$ then $L_{A}$ is compact if and only if
\begin{enumerate}
\item[(i)] for $p=\infty,$
 \[ \lim\limits_{r \rightarrow \infty}\left( \sup_{n\in\mathbb{N}_{r}} \sum_{k}\left| \widehat{a}_{nk} \right| \right)=0,\]
\item[ (ii)] for $1 <p < \infty,$
 \[\lim\limits_{r \rightarrow \infty}\left[\sup_{n\in\mathbb{N}_{r}}\left(\sum_{k}\left| \widehat{a}_{nk}\right|^{q} \right)^{1/q} \right]=0, \]
 \item[(iii)] for $p=1$
 \[
    \lim\limits_{r \rightarrow \infty}\left(\sup\limits_{n\in\mathbb{N}_{r},k\in\mathbb{N}}\left| \widehat{a}_{nk} \right| \right)=0.
 \]
\end{enumerate}
 \item[(b)] If $A \in \left( \ell_{p}^{\lambda}(\widehat{F}):\ell_{1}\right),$ then $L_{A}$ is compact if and only if
 \begin{enumerate}
 \item[(i)] for $1< p\leq \infty,$
 \[\lim_{r \rightarrow \infty}\left(\sup_{\mathbb{N}_{r}}\left\|\sum_{n\in\mathbb{N}_{r}}E_{n}\right\|_{q}\right)=0,
  \]
\item[(ii)] for $p=1,$
  \[\lim_{r\rightarrow\infty}\left(\sup_{k\in\mathbb{N}_0}\sum_{n=r}^{\infty}\left| \widehat{a}_{nk}\right|\right)=0.
  \]
 \end{enumerate}
\item [ (c) ] If $A \in \left( \ell_{p}^{\lambda}(\widehat{F}):c\right),$ then $L_{A}$ is compact if and only if
 \[ \lim\limits_{r \rightarrow \infty}\left(
                    \sup_{n\in\mathbb{N}_{r}} \|E_{n}-\widehat{\alpha} \|_{q}\right)=0,
 \]
  where  $\widehat{\alpha}=\left( \widehat{\alpha}_{k}\right)_{k\in\mathbb{N}_0} $ with $\widehat{\alpha}_{k}=\lim\limits_{n \rightarrow \infty} \widehat{a}_{nk}$ for all $k$ and the matrix $E=\left(e_{nk}\right)_{k,n\in\mathbb{N}}$ is defined by (\ref{eq400}).
 \end{enumerate}
 \end{cor}
\section{Conclusion}
We should state that although the domains of the matrices $\Lambda$ and $\widehat{F}$ in the classical sequence spaces $\ell_{p}$ and $\ell_{\infty}$ are investigated by Mursaleen and Noman \cite{M1}, and Kara \cite{Kara}, since we employ the composition of the triangles $\Lambda$ and $\widehat{F}$ the main results of the present paper are much more general than the corresponding results obtained by Mursaleen and Noman \cite{M1}, and Kara \cite{Kara}. It is worth mentioning here that in spite of the domain of the matrix $E$ in the space $\ell_{p}$ of absolutely $p$-summable sequences has been
studied in the present paper for the case $1\leq p\leq\infty,$ one can derive the similar results concerning the domain of the matrix $E$ in the space $\ell_{p}$ for $0<p<1$ which are new and are also complementary of our contribution.

\end{document}